\documentclass[11pt,a4paper]{amsart}
\usepackage[utf8]{inputenc}
\usepackage[T1]{fontenc}
\usepackage[english]{babel}
\usepackage{geometry}
\usepackage{amsmath}
\usepackage{amsfonts}
\usepackage{amssymb}
\usepackage{amsthm}
\usepackage{tikz}
\usepackage{tikz-cd}
\usetikzlibrary{patterns}

\usepackage{mathrsfs}                    % See geometry.pdf to learn the layout options. There are lots.
\usepackage{mathabx}              % Activate for for rotated page geometry

\usepackage{graphicx}
\usepackage{fullpage}
\usepackage[all]{xy}
\usepackage{xcolor}
\usepackage{hyperref}

\theoremstyle{plain}
\newtheorem*{thm*}{Theorem}
\newtheorem{thm}{Theorem}[section]
\newtheorem{prop}[thm]{Proposition}
\newtheorem{cor}[thm]{Corollary}

\newtheorem{lem}[thm]{Lemma}
\newtheorem{lemdef}[thm]{Lemma and Definition}
\newtheorem{que}[thm]{Question}
\newtheorem{ques}[thm]{Questions}

\theoremstyle{definition}
\newtheorem{de}[thm]{Definition}

\theoremstyle{example}
\newtheorem{ex}[thm]{Example}
\newtheorem{exs}[thm]{Examples}
\newtheorem{rem}[thm]{Remark}

\renewcommand{\phi}{\varphi}
\renewcommand{\rho}{\varrho}
\renewcommand{\epsilon}{\varepsilon}

\usepackage[normalem]{ulem}
\usepackage{marginnote}
\usepackage{xcolor}

\newcommand{\Hom}{\operatorname{Hom}}
\newcommand{\Ext}{\operatorname{Ext}}
\newcommand{\Tor}{\operatorname{Tor}}
\newcommand{\Gen}{\operatorname{Gen}}
\newcommand{\gen}{\operatorname{gen}}
\newcommand{\Cogen}{\operatorname{Cogen}}
\newcommand{\cogen}{\operatorname{cogen}}
\newcommand{\Prod}{\operatorname{Prod}}
\newcommand{\Add}{\operatorname{Add}}
\newcommand{\add}{\operatorname{add}}
\newcommand{\Ann}{\operatorname{Ann}}
\newcommand{\Alpha}{\operatorname{A}}
\newcommand{\Beta}{\operatorname{B}}

\newcommand{\lMod}[1]{#1\operatorname{-}\operatorname{Mod}}

\newcommand{\lmod}[1]{#1\operatorname{-}\operatorname{mod}}

\newcommand{\Cosilt}[1]{\mathbf{Cosilt}(#1)}

\newcommand{\Tpair}[2]{(\mathcal{#1}, \mathcal{#2 })}
\newcommand{\tpair}[2]{(\mathbf{#1}, \mathbf{#2 })}

\newcommand{\Acal}{\mathcal{A}}

\newcommand{\Ccal}{\mathcal{C}}
\newcommand{\Dcal}{\mathcal{D}}

\newcommand{\Fcal}{\mathcal{F}}

\newcommand{\Lcal}{\mathcal{L}}
\newcommand{\Mcal}{\mathcal{M}}

\newcommand{\Pcal}{\mathcal{P}}
\newcommand{\Qcal}{\mathcal{Q}}

\newcommand{\Scal}{\mathcal{S}}
\newcommand{\Tcal}{\mathcal{T}}
\newcommand{\Ucal}{\mathcal{U}}
\newcommand{\Vcal}{\mathcal{V}}
\newcommand{\Wcal}{\mathcal{W}}
\newcommand{\Xcal}{\mathcal{X}}
\newcommand{\Ycal}{\mathcal{Y}}

\begin{document}
\author{Lidia Angeleri H\"ugel}
\address{Lidia Angeleri H\"ugel, Dipartimento di Informatica - Settore di Matematica, Universit\`a degli Studi di Verona, Strada le Grazie 15 - Ca' Vignal, I-37134 Verona, Italy} \email{lidia.angeleri@univr.it}
\author{Francesco Sentieri}
\address{Francesco Sentieri, Dipartimento di Informatica - Settore di Matematica, Universit\`a degli Studi di Verona, Strada le Grazie 15 - Ca' Vignal, I-37134 Verona, Italy} \email{francesco.sentieri@univr.it}
\thanks{The authors  acknowledge support from the project \textit{REDCOM: Reducing complexity in algebra, logic, combinatorics}, financed by the programme  \textit{Ricerca Scientifica di Eccellenza 2018} of the Fondazione Cariverona.}
\date{\today}
\title{Wide coreflective subcategories and torsion pairs}
\maketitle
\begin{abstract}
We revisit
a construction of wide  subcategories  going back to work of Ingalls and Thomas. To a torsion pair in  the category $\lmod R$ of finitely presented modules over a left artinian ring $R$, we assign two wide subcategories in the category $\lMod R$ of all $R$-modules and describe them explicitly  in terms of an associated cosilting module.
It turns out that these subcategories are coreflective, and  we address the question of which wide coreflective subcategories can be obtained in this way.  Over a tame hereditary algebra, they are precisely the categories which are perpendicular to collections of pure-injective modules.
%When $R$ is the Kronecker algebra, this  leads us to an open problem of Krause and Stevenson concerning the classification of localizing subcategories in the derived category of quasi-coherent sheaves on the projective line: are there more localizing subcategories beyond the ones constructed from our understanding of the compact objects?
%We show that under certain set-theoretic assumptions the answer is no.
\end{abstract}

%\tableofcontents

\section{Introduction}

%%%%
%%%%%%%%%%%%%
A subcategory $\Xcal$ of the  module category $\lMod R$ over a ring $R$ is said to be reflective, respectively coreflective, if the inclusion functor $\Xcal\hookrightarrow\lMod R$ admits a left, respectively right, adjoint. A result of Gabriel and de la Pe\~na characterizes the subcategories which are both reflective and coreflective as those which arise as module categories  $\Xcal=\lMod S$ from some ring epimorphism $R\to S$. Much less is known when only one of the two conditions is satisfied, even when restricting to wide, i.e.~exact abelian, subcategories of  $\lMod R$. 

The aim of this paper is to revisit a construction of wide subcategories due to Ingalls and Thomas \cite{noncrossingPart}. We will see that these wide subcategories often turn out to be coreflective and, moreover, they are completely determined by their finitely presented modules.

Here is the construction. To a torsion pair $(\Tcal,\Fcal)$ in $\lMod R$, one associates the subcategories 
$$\alpha(\Tcal) = \left\{ X \in \mathcal{T}\,|\text{  all morphisms } f : T \to X \text{ with } T \in \mathcal{T} \text{ have } \ker(f) \in \mathcal{T}   \right\}$$
and $\beta(\Fcal)$ which is defined dually. 

When the ring $R$ is left noetherian, we  
focus on torsion pairs $(\Tcal,\Fcal)$ arising as direct limit closures of torsion pairs 
$\tpair{t}{f}$ inside the category $\lmod R$ of finitely presented modules. Such torsion pairs are parametrized by cosilting modules (Theorem~\ref{thm:CB-bij}), we denote their collection by $\mathbf{Cosilt}(R)$. We  give an explicit description of $\alpha(\Tcal)$ and $\beta(\Fcal)$ in terms of the associated cosilting module and  prove the following result.

\medskip

{\bf Theorem A} (Theorem~\ref{themapalpha}) 
%and~\ref{prop:betaInj},  Lemma~\ref{lem:betaCompHereditary})
{\it Let $ R $ be a left noetherian ring.
The construction $\alpha(\mathcal{T})$ above defines  a surjective map
$\alpha : \mathbf{Cosilt}(R) \longrightarrow \overline{\mathbf{wide}}(R)$, where  $\overline{\mathbf{wide}}(R)$ denotes the collection  of all subcategories of $\lMod R$  of the form $ \varinjlim\Wcal$ for some  wide subcategory $\Wcal$ of $\lmod R$.
}

\medskip

As a consequence, every subcategory $ \varinjlim\Wcal$ in $\overline{\mathbf{wide}}(R)$ is wide and coreflective and satisfies $ \varinjlim\Wcal\cap\lmod R=\Wcal$.
A parallel result for $ \beta$ holds true over  left artinian rings  (Theorem~\ref{prop:betaInj}).
%Let $R$ be a left noetherian ring and let $\mathbf{CWide}(R)$ denote the collection of wide coreflective subcategories of $\lMod R$. The construction above assigns to every defines  a map $\alpha : \mathbf{Cosilt}(R) \longrightarrow \mathbf{CWide}(R)$ whose image consists 
%is the class  %$\overline{\mathbf{wide}}(R)$  of the subcategories of the form $\varinjlim\Wcal$ for some wide subcategory $\Wcal$ of $\lmod R$.

%(2) The assignment $\Wcal\mapsto \varinjlim\Wcal$ yields  a one-one correspondence between wide subcategories of $\lmod R$ and subcategories in $\overline{\mathbf{wide}}(R)$, the inverse map is given by restriction to $\lmod R$.}

%(2) If $R$ is left artinian, the construction $\beta(\Fcal)$ defines   an injective map$$\beta : \mathbf{Cosilt}(R) \longrightarrow \mathbf{CWide}(R).$$ If $R$ is also hereditary, then the image of $\beta$ contains all subcategories of the form  $\Wcal^{\perp_{0,1}}$ for some wide subcategory $\Wcal$ of $\lmod R$.}

\medskip

%Besides from playing a crucial role in the proof of Theorem A, t
The properties of the maps $\alpha$ and $\beta$   lead to new insight on the lattice $\mathbf{tors}\Lambda$ of torsion classes in the category $\lmod\Lambda$ over a  finite dimensional algebra $\Lambda$. In particular, they lead to  new characterizations of $\tau$-tilting finite algebras. This class of finite dimensional algebras was introduced in \cite{g-vectors} and can be defined by a number of equivalent conditions which postulate finiteness  of certain classes of modules. For example, $\Lambda$ is   $\tau$-tilting finite if there are only finitely many isomorphism classes of finite dimensional bricks, or equivalently, only finitely many torsion pairs in $\lmod\Lambda$. We show that $\tau$-tilting finiteness can also be phrased in terms of properties of the class of wide subcategories  of $ \lMod{\Lambda} $. Here is a sample.

\medskip

{\bf Theorem B}
(Theorem~\ref{prop:tauFiniteWideCop})
The following statements are equivalent for an artin algebra $ \Lambda $.
\begin{itemize}
\item[(i)] $ \Lambda $ is $ \tau$-tilting finite.
\item[(ii)]  Every wide subcategory  of $ \lMod{\Lambda} $ closed under coproducts belongs to 
$\overline{\mathbf{wide}}(\Lambda) $.
\item[(iii)] There are only finitely many   wide subcategories of $ \lMod{\Lambda} $ closed under coproducts.
\end{itemize}

\medskip

 In the last part of the paper,
we address the question of which wide coreflective subcategories can be obtained via the maps $\alpha$ and $\beta$. 

When $\Lambda$ is the Kronecker algebra, i.e.~the path algebra $\Lambda$ of the Kronecker quiver 
{\small \begin{tikzcd}
0 \ar[r, bend left] \ar[r, bend right] & 1
\end{tikzcd}}, this  leads us to an open problem of Krause and Stevenson \cite{KSte} concerning the classification of localizing subcategories in the derived category of quasi-coherent sheaves on the projective line: are there more localizing subcategories beyond the ones constructed from our understanding of the compact objects? This question can be rephrased as follows.

\smallskip

{\bf Question:} 
Is it true that every wide coreflective subcategory $\Xcal$ of $\lMod\Lambda$ is the (left) perpendicular category ${}^{\perp_{0,1}} \Pcal$ of a collection of indecomposable pure-injective modules $\Pcal$?

\smallskip

We show  that $\Xcal$ has this shape if and only if it arises from a  wide subcategory $\Wcal$ of the category $\lmod\Lambda$ of finite dimensional $\Lambda$-modules by some standard constructions. Our result holds true over any tame hereditary algebra.

\medskip

{\bf Theorem C}
(Theorem~\ref{tamethm})
 The following statements are equivalent for a wide coreflective subcategory $\Xcal$ over a tame hereditary algebra $\Lambda$.\begin{enumerate}
\item There is a set of $\Pcal$ of indecomposable pure-injective $\Lambda$-modules such that $\Xcal={}^{\perp_{0,1}}\Pcal$.
\item There is a wide subcategory $\Wcal$ of $\lmod\Lambda$ such that  $\Xcal$ is either the (right) perpendicular category $\Wcal^{\perp_{0,1}}$ or the direct limit closure $\varinjlim \Wcal$ of $\Wcal$.
\end{enumerate}

\medskip 

We close the paper with a possible approach to the question above. The idea is to consider a family of  submodules of the generic  module $G$ over the Kronecker algebra that were constructed by  Ringel in \cite{tameWild}. They are indexed by subsets of $k$, and for infinite disjoint subsets they form large  semibricks, that is, collections of Hom-orthogonal infinite dimensional modules with endomorphism ring $k$. If $B$ is such a module, its perpendicular category ${}^{\perp_{0,1}} B$ is wide and has no indecomposable pure-injective modules.  This implies that ${}^{\perp_{0,1}} B$  contains a wide coreflective subcategory $\Xcal$ which cannot arise from a wide subcategory of $\lmod\Lambda$ as described above, unless it is zero. Unfortunately, however, we are not able to exclude the case ${}^{\perp_{0,1}} B=0$, and so the classification problem from \cite{KSte} remains unsolved.

\bigskip

{\bf Structure of the paper.} In Section 2 we collect some preliminaries on torsion pairs, purity, approximations, and cosilting theory. The constructions $\alpha(\Tcal)$ and $\beta(\Fcal)$ are introduced in Section 3. In  Section 4, we study the case when $\Tpair{T}{F}$ is in $\mathbf{Cosilt}(R)$, and we prove Theorem A together with further fundamental results on the maps $\alpha$ and $\beta$. Section 5 is devoted to applications to the lattice $\mathbf{tors}\Lambda$ over a  finite dimensional algebra $\Lambda$. We first show  that the maps $\alpha$ and $\beta$ control the shape of the Hasse quiver of $\mathbf{tors}(\Lambda)$, and more precisely, the existence of locally maximal or minimal elements in $\mathbf{tors}(\Lambda)$. The notion of a minimal cosilting module from \cite{paramTp} plays an important role in this context. Then we turn to some characterizations of $\tau$-tilting finiteness, including Theorem B, and we close the section with some open problems.
In Section 6 we focus on hereditary rings and revisit the notion of an Ext-orthognal pair from \cite{extPairs}. Section 7 is devoted to the proof of Theorem C. Finally, in Section 8, we discuss the classification problem explained above.

\bigskip

{\bf Notation.}
Given a class of objects $\Scal$ in an abelian category $\Acal$, we write
 $\Add(\mathcal{S})$  for the class of objects isomorphic to direct summands of direct sums of objects in $\Scal$, and $\Prod(\mathcal{S})$  for the class of objects isomorphic to direct summands of products of objects in $\Scal$. The class of objects isomorphic to direct summands of finite direct sums of objects in $\Scal$ is denoted by $\add(\mathcal{S})$. Moreover,
 $ \Cogen(\mathcal{S}) $ denotes the class of objects isomorphic to a subobject of a product of objects in $ \mathcal{ S } $,
and $ \Gen(\mathcal{S}) $ is defined dually.
Finally, $ {}^{\perp_{0,1}} \mathcal{S} $ is the subcategory consisting of the objects $ X \in \mathcal{A} $ such that $ \Ext^i_{\mathcal{A}}(X, S ) = 0 $ for  $  i\in\{0,1\} $ and  $ S \in \mathcal{S} $. Similarly one defines $ {}^{\perp_{0}} \mathcal{S} $, $ {}^{\perp_{1}} \mathcal{S} $, $ \mathcal{S} ^{\perp_{0,1}} $ etc.

\smallskip

Unless otherwise stated,  $R$ will denote an arbitary ring. We denote by   $ \lMod{R} $  the category of all left $R$-modules and by $ \lmod{R} $  the category of finitely presented left $R$-modules.
If  $ \mathcal{S} $ is a class of modules in $ \lMod{R} $, we denote by $ \varinjlim \mathcal{ S } $ the full subcategory of $ \lMod{R} $ whose objects are the colimits of directed systems of modules in $ \mathcal{S} $. 
When $ \mathcal{S} $ is a class of finitely presented modules closed under finite direct sums,  $ \varinjlim \mathcal{ S } $ is closed under directed colimits by \cite[Proposition 2.1]{Lenzing}. This is not true for a general $ \mathcal{S} $, see \cite[Example 1.1]{directLimitFinProj}.

\section{Torsion pairs}
\label{chp:Prelim}
In this section we fix the terminology and collect some fundamental concepts and results that we will use in the sequel. We start out reviewing the notion of a torsion pair. Then we focus on torsion pairs whose torsionfree class is closed under direct limits and describe them in terms of approximation theory and cosilting theory. This allows us to show that the torsion pairs  in  the category  of finitely presented modules over a left noetherian ring are parametrized by cosilting modules.

\begin{de}{\rm Let $ \mathcal{A} $ be an abelian category.

(1) Two subcategories $ \mathcal{T} $, $ \mathcal{F} $  of $ \mathcal{A} $ form   a \emph{torsion pair} $ ( \mathcal{T}, \mathcal{F} ) $  if:
\begin{itemize}
\item[(i)] For all $ F \in \mathcal{F} $, for all $ T \in \mathcal{T} $, $ \Hom_{\mathcal{A}}(T, F) = 0 $.
\item[(ii)] For all $ M \in \mathcal{A} $ there is a short exact sequence
\[
0 \to T \to M \to F \to 0
\]
with $ T \in \mathcal{T} $ and $ F \in \mathcal{F} $.
\end{itemize}
In this case, $ \mathcal{T}$ is a \emph{torsion class} and $ \mathcal{F} $ is a \emph{torsionfree class}.

(2) Let  $ \mathcal{A} $ be a complete and co-complete abelian category. Given a class of objects $ \mathcal{C} $  in $\Acal$, we can  form 
\begin{itemize}
\item[-]
the torsion pair  $ (  \mathbf{T}(\mathcal{C}), \mathcal{C}^{\perp_0}    ) $ \emph{generated} by $ \mathcal{C} $, and 
\item[-]
the torsion pair $ ( {}^{\perp_0}\mathcal{C}, \mathbf{F}(\mathcal{C})   ) $ \emph{cogenerated} by $ \mathcal{C} $.
\end{itemize}
Here $ \mathbf{T}(\mathcal{C}) ={}^{\perp_0}( \mathcal{C}^{\perp_0} )$
is the smallest torsion class containing $\Ccal$, and 
 $\mathbf{F}(\mathcal{C})= ({}^{\perp_0} \mathcal{C})^{\perp_0}  $  is the smallest  torsionfree class containing $ \mathcal{C} $.}
 \end{de}

In the case of module categories, there are  well-known explicit descriptions for the torsion resp. torsionfree class generated by $\Ccal$. We first need the following easy observation.

\begin{lem}\label{cor:torsionNoetherianCat}
If $ \mathcal{A} $ is an abelian category such that all the objects of $ \mathcal{A} $ are noetherian, then a subcategory of $  \mathcal{A} $ is a torsion class if and only if it is closed under extensions and quotients. 

Dually, if $ \mathcal{A} $ is an abelian category such that all the objects of $ \mathcal{A} $ are artinian, then a subcategory of $  \mathcal{A} $ is a torsionfree class if and only if it is closed under extensions and submodules. \end{lem}

Given a left coherent ring $R$, we consider the abelian category $\lmod R$ and denote by $ \mathbf{tors}(R) $  the collection of all torsion pairs in $ \lmod{R} $.
We use the symbols $ \widetilde{\mathbf{T}}(\mathcal{C}) $ and $ \widetilde{\mathbf{F}}(\mathcal{C}) $ for the  torsion resp.~torsionfree class in $\lmod R$ generated by some subcategory $ \mathcal{C} $.
We further denote by $ \gen(\mathcal{C}) $  the class of objects isomorphic to a quotient of a finite direct  sum of objects in $ \mathcal{ C } $, define dually $\cogen\Ccal$, and write $ \operatorname{filt}(\mathcal{C}) $ for the extension closure of $ \mathcal{C} $.

\begin{prop}{\cite[Lemma 3.1]{wideLoc}} \label{prop:torsionTorsionfreeClassSmall}
Let $ R $ be a ring and $ \mathcal{C} $ a  subcategory of $ \lmod{R} $.

(1) If $R$ is left noetherian, then 
$\widetilde{\mathbf{T}}(\mathcal{C}) = \operatorname{filt}\gen(\mathcal{C}).$

(2) If $R$ is left artinian, then 
$\widetilde{\mathbf{F}}(\mathcal{C}) = \operatorname{filt}\cogen(\mathcal{C})$.
\end{prop}

Using transfinite extensions, one can obtain an analogous description of the torsion class $\mathbf{T}(\mathcal{C})$ in the module category $\lMod{R} $, see
{\cite[Lemma 3.2]{tau-tilt}}.

\medskip

Next, we collect some well known facts about definable classes and purity. A comprehensive reference can be found in \cite{purity}. 

\begin{de}{\rm
(1) A short exact sequence $ 0 \to L \to M \to N \to 0 $ in $ \lMod{R} $ is \emph{pure-exact} if for every $ U \in \lmod{R} $ the sequence
\[
\begin{tikzcd}
0 \arrow[r] & \mathrm{Hom}(U, L) \arrow[r] & \mathrm{Hom}(U, M) \arrow[r] & \mathrm{Hom}(U, N) \arrow[r] & 0
\end{tikzcd}
\]
is an exact sequence of abelian groups. In this case, we say that $ L $ is a \emph{pure submodule} of $ M $ or that the map $ L \to M $ is a \emph{pure monomorphism}. }

(2) A module $ E $ is \emph{pure-injective} if every pure exact sequence starting at $ E $ is split exact.

(3)
A subcategory of $\lMod{R} $ is \emph{definable} if it is closed under products, pure submodules and direct limits.
%Let $ \mathcal{C} \subseteq \lMod{R} $. We denote by $ \langle \mathcal{C} \rangle $ the \emph{definable subcategory generated} by $ \mathcal{C} $, that is the smallest definable subcategory of $ \lMod{R} $ containing $ \mathcal{C} $.}
\end{de}

Notice that a torsionfree class in $ \lMod{R} $ is definable if and only if it is closed under direct limits, as closure under products and (pure) submodules is granted.
Moreover, definable torsionfree classes can be described in terms of approximation theory and cosilting theory. Let us recall the relevant notions.

%%%%%%%%%%%%%%%%%%%%%%%%%%
\begin{de}{\rm Let $\Acal$ be an abelian category with a subcategory  $ \mathcal{S} \subseteq \mathcal{A} $. Let $ M \in \mathcal{A} $. A morphism $ g : S \to M $ is a \emph{$ \mathcal{S}$-precover} if $ S \in \mathcal{S} $ and  every morphism $ g' : S' \to M $ with  $ S' \in \mathcal{S} $ factors through $g$. The map
$ g $ is a \emph{$ \mathcal{S}$-cover} if in addition it is right minimal, i.e.~every endomorphism $s$ of $S$ such that $gs=g$ is an isomorphism. 
Finally, the subcategory $ \mathcal{S} $ is called \emph{(pre)covering} if every object in $ \mathcal{A} $ admits an $ \mathcal{S}$-(pre)cover. 
Dually, we define \emph{$ \mathcal{S}$-(pre)envelopes} and \emph{(pre)enveloping} subcategories.}

When  $\Acal=\lmod \Lambda$ for an artin algebra $\Lambda$, then subcategories which are both precovering and preenveloping (and therefore covering and enveloping)
are called \emph{functorially finite}. \end{de}
%%%%%%%%%%%%%%%%%%%%%%%
%For artin algebras, we have the following: 
%\begin{prop}[{\cite[Section 1.F]{Ringel}}]\label{prop:pureForArtinAlgebra}Let $ \Lambda $ be an artin algebra. A short exact sequence $ 0 \to L \to M \to N \to 0 $ in $ \lMod{\Lambda} $ is \emph{pure} if for every $ U \in \lmod{\Lambda} $ the sequence\[\begin{tikzcd} 0 \arrow[r] & \mathrm{Hom}(N, U) \arrow[r] & \mathrm{Hom}(M, U) \arrow[r] & \mathrm{Hom}(L, U) \arrow[r] & 0\end{tikzcd}\] is an exact sequence of abelian groups. \end{prop}
%%%%%%%%%%%%%%%%%%%%%%%%%%
\begin{de}{\rm
(1) We say that an $R$-module $ C $ is \emph{cosilting} if there exists an injective copresentation $ \omega : I_0 \to I_1 $ such that:
\[
\Cogen(C) = \mathcal{C}_\omega:= \Bigl\{ X \in \lMod{R}\,\Big|\, \Hom_R(X, \omega)\ \text{is surjective}  \Bigr\}
\]

(2) Two cosilting modules $ C_1, C_2 $ are \emph{equivalent} if $ \Cogen(C_1) = \Cogen(C_2) $. 
%This holds true if and only if  $ \Prod(C_1) = \Prod(C_2) $.

(3) A module $ C $ is \emph{cotilting} (of injective dimension at most one) if $ \Cogen(C) = {}^{\perp_1}C $, or equivalently, if it is cosilting with respect to an injective copresentation which is surjective. 

\smallskip

 \emph{Silting modules} and \emph{tilting modules} of projective dimension at most one are defined dually.}
\end{de}

We collect some important properties of cosilting modules.
 
\begin{thm}
\label{thm:cosiltCover}
(1) {\cite[Theorem 4.7]{cosiltMod}}
Every cosilting module is pure-injective.

(2) \cite{torsClassGenSilt,cosilt}, \cite[Theorem 3.8 and Corollary 3.9]{abundance}
A torsionfree class $ \mathcal{F} \subseteq R-\mathrm{Mod} $ is definable if and only if it is covering, if and only if $ \mathcal{F} = \Cogen(C) $ for some cosilting module $ C $.
\end{thm}

\medskip

In light of the theorem above, we will denote by $ \mathbf{Cosilt}(R) $ the collection  of all torsion pairs with definable torsionfree class, and refer to such pairs as \emph{cosilting torsion pairs}. 
%%%%%%%%%%%%%%%%%
The interplay between torsion pairs in $\lmod R$ and cosilting torsion pairs is based on the following fundamental result which goes back to \cite{locallyFp}.

%We also fix the following notation:\begin{de}Let $ \mathcal{C} $ be a class of modules in $ \lMod{R} $. We denote by $ \varinjlim \mathcal{ C } $ the full subcategory of $ \lMod{R} $ whose objects are the colimits of directed systems of modules in $ \mathcal{C} $. \end{de} \begin{rem}If $ \mathcal{C} $ is a class of finitely presented modules closed under finite direct sums, then by a result of Lenzing, see \cite[Section 4.1]{locallyFp} for a proof, the category $ \varinjlim \mathcal{ C } $ is closed under directed colimits. This is not true for a general $ \mathcal{C} $, see \cite[Example 1.1]{directLimitFinProj}.\end{rem}

\begin{thm}
\label{thm:CB-bij}
When $ R $ is a left noetherian ring, there is a bijection \[
\mathbf{tors}(R) \leftrightarrow \mathbf{Cosilt}(R).
\]
It associates to a torsion pair $ ( \mathbf{t}, \mathbf{f} ) $ in $ \lmod{R} $ the direct limit closure $ ( \mathcal{T},\Fcal) := (\varinjlim \mathbf{t}, \varinjlim \mathbf{f}) $, which coincides with  the torsion pair $(\Gen\mathbf t,  \mathbf{t}^{\perp_0})$ generated by $\mathbf t$.
The inverse of this map sends a cosilting torsion pair $ ( \mathcal{T} , \mathcal{F} ) $ to its restriction $ ( \mathcal{T} \cap \lmod{R}, \mathcal{F} \cap \lmod{R} ) $.
\end{thm}

% \reml{Per\`o \cite{tau-orig} \`e per algebre fin.dim. Vedi anche Section  5.}

In other words, the torsion pairs in $\mathbf{tors}(R)$ over a left noetherian ring are parametrized by cosilting modules. Over an artin algebra, the finitely generated cosilting modules are precisely the  $ \tau^{-1}-$tilting modules. So, we can use the results from \cite{tau-orig} to observe the following. 
\begin{rem}\label{ff} {\rm Assume that $\Lambda$ is an artin algebra and let $\tpair{t}{f}$ be a torsion pair in  $\mathbf{tors}(\Lambda)$. Then $\mathbf{t}$ is functorially finite if and only if so is $\mathbf{f}$, and this happens precisely when the associated cosilting module is finitely generated.}
\end{rem}

\begin{rem}
{\rm In the literature, $ \tau-$tilting theory is usually applied in the case of finite-dimensional algebras $ \Lambda $.
 However, all the results which we will use from \cite{tau-orig} , \cite{g-vectors} and \cite{nagoya} are valid in the more general setting of artin algebras.
 Two crucial points to ensure the validity of such results are the fact that $ \lmod{\Lambda} $ is an abelian length category and that for every torsion pair $ \tpair{t}{f} $ in $ \lmod{\Lambda} $ the torsion class $ \mathbf{ t } $ is functorially finite if and only if the torsionfree  class $ \mathbf{f} $ is functorially finite. 
 }
\end{rem}

%%%%%%%%%%%%%%
\section{Wide subcategories}

We now introduce the construction of wide subcategories from torsion pairs due to Ingalls and Thomas \cite{noncrossingPart}. From the interplay between small and large torsion pairs over a noetherian ring $R$ we derive  some compatibility results between the constructions in $\lmod R$ and $\lMod R$. Furthermore, we show that the simple objects in the wide subcategories we obtain from a torsion pair are precisely the objects studied in \cite{simpleCotilting, mutation} and \cite{minimIncl}.

 \begin{de}{\rm
 Let $ \mathcal{ A } $ be an abelian category. A subcategory $ \mathcal{W} \subseteq \mathcal{A} $ is \emph{wide} if it is closed under kernels, cokernels and extensions.}
 \end{de}

 \begin{de}{\rm
 Let $ \Tpair{T}{F} $ be a torsion pair in some abelian category $ \mathcal{A} $. We define:
\begin{align*}
\Alpha(\mathcal{T}) & = \left\{ X \in \mathcal{A}\,|\, \text{ for all } T \in \mathcal{T}, f : T \to X, \ker(f) \in \mathcal{T}   \right\} \\
\Beta(\mathcal{F}) & = \left\{ X \in \mathcal{A} \,|\, \text{ for all } F \in \mathcal{F}, f : X \to F, \operatorname{coker}(f) \in \mathcal{F}   \right\} \\
\alpha(\mathcal{T}) & = \mathcal{T} \cap \Alpha(\mathcal{T}) \\
\beta(\mathcal{F}) & = \mathcal{F} \cap \Beta(\mathcal{F})
\end{align*}
}
 \end{de}
 
% \begin{rem} According to Definition \ref{de:coherentCocoherent}, the category $ \alpha(\mathcal{T}) $ consists of the coherent objects of the torsion class. Dually, the category $ \beta(\mathcal{ F }) $ consists of the co-coherent objects of the torsionfree class. \end{rem}
 
 \begin{lem}
 \label{lem:basicAlphaBeta}
Let  $ \Tpair{T}{F} $ be a torsion pair in $ \mathcal{A} $. The following statements hold true.
\begin{itemize}
\item[(i)] The subcategory $ \Alpha(\mathcal{T}) $ is closed under subobjects and extensions. Moreover, $ \mathcal{F} \subseteq \Alpha(\mathcal{T}) $.
%\item[(ii)] The subcategory $ \Beta(\mathcal{F}) $ is closed under quotients and extensions. Moreover, $ \mathcal{T} \subseteq \Beta(\mathcal{F}) $.
\item[(ii)] $\alpha(\mathcal{T}) $ is a wide subcategory of $ \mathcal{A} $. It is closed under torsion subobjects.
%\item[(iv)] $\beta(\mathcal{F}) $ is a wide subcategory of $ \mathcal{A} $. It is closed under torsionfree quotients.
\item[(iii)]  $ \Alpha(\mathcal{ T }) $ consists of the objects $M$ of $\Acal$ appearing in  short exact sequences $0 \to C \to M \to D \to 0$ with $ C \in \alpha(\mathcal{ T }) $ and $ D \in \mathcal{F} $.
%\item[(vi)] $ \Beta(\mathcal{F}) = \mathcal{T} \star \beta(\mathcal{F}) $ 
%\item[(iv)] 
In particular, $ \alpha(\mathcal{T}) = 0 $ if and only if $ \Alpha(\mathcal{T}) = \mathcal{F} $.
%\item[(viii)]  $ \beta(\mathcal{F}) = 0 $ if and only if $ \Beta(\mathcal{F}) = \mathcal{T} $.
\end{itemize}
The dual statements hold true for $ \Beta(\mathcal{F}) $ and $\beta(\mathcal{F}) $.
 \end{lem}
 
 \begin{proof}
 %We give a proof of (i), (iii), (v) and (vii), the other points having similar proofs.
 (i) First, notice that if $ F \in \mathcal{F} $, then $ \Hom_R(T, F) = 0 $ for every $ T \in \mathcal{T} $. Thus, every such object $ F $ is trivially an element of $ \Alpha(\mathcal{T}) $. 
 
 Let $ X \in \Alpha(\mathcal{T}) $, consider $ Y \le X $. Then for every torsion object $ T $ and every map $ f : T \to Y $, the kernel of $ f $ is equal to the kernel of the composition of $ f $ with an embedding of $ Y $ into $ X $. Thus, $ Y $ is in $ \Alpha(\mathcal{T}) $.
 
 Let $\begin{tikzcd}
 0 \arrow[r] & X' \arrow[r] & Y \arrow[r, "g"] & X'' \arrow[r] & 0  
 \end{tikzcd}$ be a short exact sequence with $ X', X'' \in \Alpha(\mathcal{ T }) $. Let $ f : T \to Y $ be some map. 
 Consider the following commutative diagram:
 \[
 \begin{tikzcd}
  0 \arrow[r] & \ker(g \circ f) \arrow[r] \arrow[d, dashed, "h"] & T \arrow[r] \arrow[d,"f"] & \mathrm{Im}(g \circ f) \arrow[r] \arrow[d, hook] & 0 \\  
  0 \arrow[r] & X' \arrow[r] & Y \arrow[r, "g"] & X'' \arrow[r] & 0  
  \end{tikzcd}
 \]
 By the previous point, $ \mathrm{Im}(g \circ f) \in \Alpha(\mathcal{ T }) $, thus $  \ker(g \circ f) \in \mathcal{T} $. An application of the snake lemma yields that $ \ker(f) = \ker(h) $ and this is a torsion module since $ X' \in \Alpha(\mathcal{T}) $. Thus, $ \Alpha(\mathcal{ T }) $ is closed under extensions.
 
(ii) is \cite[Proposition 2.12]{noncrossingPart}, and
(iii) 
%and (iv) are 
is left to the reader. \qedhere

 % $ \Alpha(\mathcal{ T }) \supseteq \alpha(\mathcal{ T }) \star \mathcal{F} $ is immediate as by the previous points we have that $  \Alpha(\mathcal{ T }) \supseteq \mathcal{F} $ and that it is closed under extensions.
 
 %For the reverse inclusion, given $ X \in \Alpha(\mathcal{ T }) $, the approximation sequence $ 0 \to tX \to X \to X/tX \to 0 $ with respect to the torsion pair $ \Tpair{T}{F} $ has the required shape.
 
 %\item[(iv)] Immediate by point (v). 
 \end{proof}
 %If $ \mathcal{A} $ is a small category in some sense, e.g. if it is the category of finitely presented objects in some locally noetherian Grothendieck category $ \mathcal{G} $, then we will use the symbols $ \widetilde{\Alpha} $, $ \widetilde{\Beta} $, $ \widetilde{\alpha} $ and $ \widetilde{\beta} $ for these operators in $ \mathcal{A} $.
 We want to study these costructions when $\Acal$ is a module category.
 Over a left coherent ring $R$, besides  $\lMod R$, we can also consider the abelian category $\lmod R$. We will use the symbols $ \widetilde{\Alpha} $, $ \widetilde{\Beta} $, $ \widetilde{\alpha} $ and $ \widetilde{\beta} $ for the operators in $\lmod R$.
If $R$ is left noetherian, we can use the interplay between torsion pairs in $\lmod R$ and cosilting torsion pairs in $\lMod R$ to obtain the following compatibility result:

\begin{lem}
\label{prop:cosiltRestriction}
Let $ R $ be a left noetherian ring. Let $ \Tpair{T}{F} $ be a cosilting torsion pair in $ \lMod{R} $ with restriction $(\mathbf{t},\mathbf{f})$ to $ \lmod{R} $. The following statements hold true.
\begin{itemize}\item[(i)] $ \Alpha(\mathcal{T}) \cap \lmod{R} = \widetilde{\Alpha}(\mathbf{t}) $ and thus  $ \alpha(\mathcal{T}) \cap \lmod{R} = \widetilde{\alpha}(\mathbf{t}) $\item[(ii)]
 $ \Beta(\mathcal{F}) \cap \lmod{R} = \widetilde{\Beta}(\mathbf{f}) $ and thus $ \beta(\mathcal{F}) \cap \lmod{R} = \widetilde{\beta}(\mathbf{f}) $.
\end{itemize} 
\end{lem}

\begin{proof}
%We can adapt the proof of \ref{prop:finInf} to this situation. 
(i) The inclusion $ \Alpha(\mathcal{T}) \cap \lmod{R} \subseteq \widetilde{\Alpha}(\mathbf{t}) $ is immediate. 
Assume $ A \in \widetilde{\Alpha}(\mathbf{ t }) $. Let $ f : T \to A $ be a morphism with $ T \in \mathcal{T} $. We need to show that $ K := \ker f \in \mathcal{T} $. As $ \widetilde{\Alpha}(\mathbf{t}) $ is closed under submodules, we may assume that $ f $ is an epimorphism.
Since the torsion pair is cosilting, we can find a family of finitely-generated torsion modules $ \{t_i\}_I $ with an epimorphism $ p : \bigoplus_I t_i \to T $.  
At this point, since $ A $ is finitely-generated, we can find a finite subset $ J \subseteq I $, such that the map $ f_J := f \circ (p|_J) $ is an epimorphism. Denote by $ K_J $ the kernel of this map and consider  the following commutative diagram:
\[
\begin{tikzcd}
0 \ar[r] & K_J \ar[d] \ar[r] & \bigoplus_J t_i \ar[d, "p|J" ] \ar[r, "f_J"] & A \ar[d, equals] \ar[r] & 0 \\
0 \ar[r] & K \ar[d, two heads] \ar[r] & T \ar[r, "f"] \ar[d, two heads] & A \ar[r] & 0 \\
& L \ar[r,equals] & L 
\end{tikzcd}
\]
We have that $ L \in \mathcal{T} $, being the quotient of a torsion module. Moreover, $ K_J \in \mathbf{t} $ as $  \bigoplus_J t_i \in \mathbf{t} $ and $ A \in \widetilde{\Alpha}(\mathbf{t}) $. Therefore, $ K \in \mathcal{T} $ as required.

(ii)
 Again $ \Beta(\mathcal{F}) \cap \lmod{R} \subseteq \widetilde{\Beta}(\mathbf{f}) $ by definition. So assume $ X \in \widetilde{\Beta}(\mathbf{f}) $, let $ F \in \mathcal{F} $ and $ f : X \to F $ with cokernel $ C $. To show that $ X \in \Beta(\mathcal{F}) $ we need to prove that $ C $ is torsionfree. As $ \mathcal{F} = \mathbf{t}^{\perp_0} $, assume we have an injection $ T \to C $, with $ T \in \mathbf{ t } $, and consider the following pull-back diagram:
\[
\begin{tikzcd}
X \arrow[d, equals] \arrow[r] & P \arrow[r] \arrow[d,hook] & T \arrow[d, hook] \arrow[r] & 0 \\
X \arrow[r, "f"] & F \arrow[r] & C \arrow[r] & 0
\end{tikzcd}
\]
By construction, $ P $ is a finitely generated torsionfree module, thus, using that $ X \in  \widetilde{\Beta}(\mathbf{f}) $ we must have that $ T \in \mathbf{f} $. Thus $ T = 0 $ being both torsion and torsionfree. 
\end{proof}

Given a left coherent ring $R$, we denote by $ \mathbf{wide}(R) $  the collection of wide subcategories of $ \lmod{R} $. We then have maps $\widetilde{\alpha}, \widetilde{\beta}: \mathbf{tors} R\to  \mathbf{wide}(R)$. 
In \cite{wideTorsion,wideLoc}, these maps are shown to be surjective over artin algebras. In fact, one just needs the description of the operators $ \widetilde{\mathbf{T}} $ and $ \widetilde{\mathbf{F}} $ in Proposition~\ref{prop:torsionTorsionfreeClassSmall}.
\begin{thm}[\cite{wideTorsion}, \cite{wideLoc}]
\label{thm:wideGeneratedCogen}
Let $ R $ be a ring and $ \mathcal{W} $  a  subcategory of $ \lmod{R} $. 
\begin{itemize}
\item[(i)] If $ R $ is a left noetherian ring, then $\widetilde{\alpha}( \widetilde{\mathbf{T}}(\mathcal{W}) ) = \mathcal{W} $ if and only if $ \mathcal{W} $ is a wide subcategory.
\item[(ii)] If $ R $ is a left artinian ring, then $\widetilde{\beta}( \widetilde{\mathbf{F}}(\mathcal{W}) ) = \mathcal{W} $ if and only if $ \mathcal{W} $ is a wide subcategory.
\end{itemize}
\end{thm}

%A ``large'' partial version of this result was obtained for torsion classes:

%\begin{prop}[\cite{tau-tilt}]\label{prop:partialWideCoprodAlpha}Let $ R $ be a ring. Let $ \mathcal{W} $ be a subcategory of $ \lMod{R} $. Then \[    \alpha(\mathbf{T}(\mathcal{W})) = \mathcal{W} \text{ if } \mathcal{W} \text{ is a wide subcategory closed under coproducts.}\]\end{prop} 

%We will prove the equivalence of the two conditions above for cosilting torsion pairs in Proposition \ref{prop:AlfaCoprodClsdIffCosiltWidelyGen}. 

%We conclude this section giving a name to such torsion pairs
\begin{de}[\cite{wideTorsion}]{\rm
Let $ R $ be a left noetherian ring. A torsion pair $ \Tpair{T}{F} $  in $ \lMod{R} $, respectively its restriction $ \tpair{t}{f} $ in $ \lmod{R} $, is said to be \emph{widely generated} if there exists a wide subcategory $ \mathcal{W} \in\mathbf{wide} (R)$  such that $ \mathbf{T}(\mathcal{W}) = \mathcal{T} $, or equivalently, $ \widetilde{\mathbf{T}}(\mathcal{W}) = \mathbf{t} $.}
\end{de}

The next result computes $\beta$ for widely generated torsion pairs over noetherian  hereditary rings.

\begin{lem}
\label{lem:betaCompHereditary}
Let $ R $ be a left noetherian ring. If $ \mathcal{W} $ is a wide subcategory of $ \lmod{R} $ consisting of modules of projective dimension less than one, then
\[
\beta(\mathcal{W}^{\perp_0}) = \mathcal{W}^{\perp_{0,1}}.
\]  
\end{lem}
\begin{proof}
Notice that $ \mathbf{T}(\mathcal{W}) = \varinjlim \widetilde{\mathbf{T}}(\mathcal{W}) $, thus $ \mathcal{W} \subseteq \alpha(\mathbf{T}(\mathcal{W})) $ by Lemma \ref{prop:cosiltRestriction} and Theorem~\ref{thm:wideGeneratedCogen}(i).

We can immediately verify that a module $ M $ in $ \mathcal{W}^{\perp_1} $ is in $ \Beta(\mathcal{W}^{\perp_0}) $ once we have noticed that it is enough to check the condition for injective maps $ 0 \to M \to F $ with $ F \in \mathcal{W}^{\perp_0} $ (here we need that $ \mathcal{W}^{\perp_1} $ is closed under quotients).

For the other inclusion, take $ M \in \beta(\mathcal{W}^{\perp_0}) = \mathcal{W}^{\perp_0} \cap \Beta(\mathcal{W}^{\perp_0}) $.
Consider a short exact sequence:
\[
0 \to M \to N \to W \to 0
\]
with $ W \in \mathcal{W} $. Then, taking the torsion part of $ N $ we obtain the following commutative diagram:
\[
\begin{tikzcd}
0 \arrow[r] & K \arrow[r] \arrow[d] & \operatorname{t}N \arrow[d, hook] \arrow[r] & I \arrow[d, hook] \arrow[r] & 0 \\
0 \arrow[r] & M \arrow[d, two heads] \arrow[r] & N \arrow[r] \arrow[d, two heads] & W \arrow[r] \arrow[d, two heads] & 0 \\
0 \arrow[r] & L \arrow[r] & \overline{N} \arrow[r] & \widetilde{W} \arrow[r] & 0
\end{tikzcd}
\]
Then since $ M \in \Beta(\mathcal{W}^{\perp_0}) $ and this class is closed under quotients, $ L \in \Beta(\mathcal{W}^{\perp_0}) $. Therefore, $ \widetilde{W} $ is in $ \mathcal{W}^{\perp_0} $ as $ \overline{N} \in \mathcal{W}^{\perp_0} $. But $ \widetilde{W} $ is a quotient of $ W $, thus it must be zero. 

Thus $ I = W $, and since $ \mathcal{W} \subseteq \alpha(\mathbf{T}(\mathcal{W})) $ we have that $ K \in \mathbf{T}(\mathcal{W}) $. However, $ K $ is also a submodule of $ M $ which is in $ \mathcal{W}^{\perp_0} $, therefore, $ K = 0 $ as it is both torsion and torsionfree. This shows that the middle sequence splits.
\end{proof}

Finally, we determine the simple objects in the wide subcategories given by $\alpha$ and $\beta$.

\begin{de}
\label{de:TtF}
{\rm
Let $ (\mathcal{ T }, \mathcal{F} ) $ be a torsion pair in $ \lMod{R} $. A non-zero module $ B \in\Tcal $ is \emph{torsion, almost torsionfree}  with respect to $ (\mathcal{T}, \mathcal{F} ) $ if it satisfies the following conditions.
\begin{itemize}
\item[(1)] Every proper submodule of $ B $ is contained in $ \mathcal{ F } $. 
\item[(2)] For every short exact sequence $ 0 \to K \to T \to B\to 0 $, if $ T \in \mathcal{T} $, then $ K \in \mathcal{T} $.
\end{itemize} 
Dually, we  define  \emph{torsionfree, almost torsion} modules.}
\end{de}
These concepts were introduced in \cite{slides} and studied in \cite{simpleCotilting,mutation}. They are closely related to the notions of  minimal extending, resp.~coextending, modules appearing in  \cite{minimIncl}.

\begin{rem}\cite[Proposition 2.11]{nagoya} 
\label{prop:finInf}
{\rm Let $ R $ be a left noetherian ring.
Let $ (\mathbf{ t }, \mathbf{ f }) $ be a torsion pair in $\lmod{R}$ and $ ( \mathcal{T}, \Fcal) $ the corresponding cosilting torsion pair in $ \lMod{R} $.
Then  all torsion, almost torsionfree modules with respect to $ (\mathcal{T}, \mathcal{F}) $ are finitely generated and coincide with the minimal co-extending modules with respect to $ (\mathbf{ t }, \mathbf{ f }) $ in the sense of  \cite{minimIncl}.
Moreover, the finitely generated torsionfree, almost torsion modules with respect to $ (\mathcal{T}, \mathcal{F}) $ are precisely the minimal extending  modules with respect to $ (\mathbf{ t }, \mathbf{ f }) $}. 
\end{rem}

\begin{prop}
\label{prop:simpInalphabeta}
Let $ R $ be a ring. Let $ \Tpair{T}{F} $ be a torsion pair in $ \lMod{R} $. Then:
\begin{itemize}
\item[(i)] The simple objects of $ \alpha(\mathcal{T}) $ are precisely the torsion, almost torsionfree modules in $ \mathcal{ T } $.
\item[(ii)] The simple objects of $ \beta(\mathcal{F}) $ are precisely the torsionfree, almost torsion modules in $ \mathcal{ F } $.
\end{itemize}
\end{prop}

\begin{proof}
We give a proof of (i). Notice that condition (2) in Definition~\ref{de:TtF} states that $B$ belongs to $\Alpha(\Tcal)$. Hence $B$  is torsion, almost torsionfree  if and only if it is an object of $\alpha(\Tcal)$ and all proper submodules of $ B $ are contained in $ \mathcal{ F } $. This clearly implies that $B$ is a simple object in  $\alpha(\Tcal)$. 
Also the reverse implication follows 
 immediately, as $ \alpha(\mathcal{T}) $ is closed under torsion submodules.\end{proof}

%Over an artin algebra, one can prove the analogous result for the silting case by dual arguments. 

%\begin{prop}\label{prop:siltRestriction}Let $ \Lambda $ be an artin algebra. Let $ \Tpair{T}{F} $ be a silting torsion pair in $ \lMod{\Lambda} $ with restriction $ \tpair{t}{f} $ to $ \lmod{\Lambda} $. Then:\begin{itemize}\item[(i)] $ \Alpha(\mathcal{T}) \cap \lmod{\Lambda} = \widetilde{\Alpha}(\mathbf{t}) $ and thus  $ \alpha(\mathcal{T}) \cap \lmod{\Lambda} = \widetilde{\alpha}(\mathbf{t}) $\item[(ii)] $ \Beta(\mathcal{F}) \cap \lmod{\Lambda} = \widetilde{\Beta}(\mathbf{f}) $ and thus $ \beta(\mathcal{F}) \cap \lmod{\Lambda} = \widetilde{\beta}(\mathbf{f}) $\end{itemize} \end{prop}

%%%%%%%%%%%%%%%%%%%%%%%%%%%%%%%%%%%%%%%%%%%%%%%%%%%%%%%%%%%%%%%%%%%%%%%%%%%%%%%%%%%%%%%%%%%%%%%%%%%%%%%

\section{Wide subcategories and cosilting modules}
\label{sec:wideFromCosilt}

For cosilting torsion pairs we can obtain an explicit description of the wide subcategories defined in the previous section. 
%For the necessary preliminaries about cosilting modules we refer back to section \ref{sec:siltCosilt}.
%\subsection{Wide subcategories from torsionfree classes}
Let $C$ be a cosilting module and $ (\Tcal, \Fcal)=({}^{\perp_0}C, \Cogen C)$ its cosilting torsion pair. Recall that every module 
 admits a $ \Fcal-$cover with kernel in $ \Prod(C) $. Let us fix an injective cogenerator $ E(R) $ with a minimal approximation sequence
\begin{equation}\label{approxsequence}
\begin{tikzcd}
0 \arrow[r] & C_1 \arrow[r] & C_0 \arrow[r,"g"] & E(R)
\end{tikzcd}
\end{equation}

\begin{lem}
\label{lem:approxSeqInFactor}
Let $ C $ be a cosilting module with approximation sequence (\ref{approxsequence})  and $ \Fcal=\Cogen C$. The following statements hold true.

(1) 
$C_0$ is split-injective in $\Fcal$, i.e.~every monomorphism $ C_0\to F $ with $F$ in $ \Fcal$  is a split monomorphism. Moreover, $ C_0 \oplus C_1 $ is a cosilting module equivalent to $ C $. 

(2) $\operatorname{Im}(g) = \left\{ x \in E(R) \,|\, \Ann(C)\,x = 0 \right\}$ is an injective cogenerator of $ R/\Ann(C) $. 

(3) $ \Fcal = {}^{\perp_1} C_1 \cap \lMod{R/\Ann(C)} $.
\end{lem}

\begin{proof}
(1) is shown in \cite[Lemma 3.3 and Theorem 3.5]{torsClassGenSilt}.

(2)
Notice that $ C $ is a cotilting module over $ R/\Ann(C)$, see \cite[Theorem 3.6]{abundance}. Thus every $ R/\Ann(C)-$module admits a surjective $ \Fcal$-cover, and clearly all modules with a surjective cover are in $ \lMod{R/\Ann(C)} $ (being annihilated by $ \Ann(C) $).

Let $ E = \left\{ x \in E(R) \,|\, \Ann(C)\,x = 0 \right\} $. This is the largest submodule of $ E(R)$ belonging to $\lMod{ R/\Ann(C) } $. Thus $ \operatorname{Im}(g) \subseteq E $. 
On the other hand, as recalled above, $ E $ admits a surjective $ \Fcal$-cover $ C' \to E $. The induced map $ C' \to E \to E(R) $ must factor through $ g : C_0 \to E(R) $ showing that $ E \subseteq \operatorname{Im}(g) $.
Now, for any $ M \in \lMod{R/\Ann(C)} $ there is a set $ I $ such that $ M $ embeds in $ E(R)^I $, but since $ M $ is in $ \lMod{R/\Ann(C)} $ this embedding must factor through $ E^I $. This shows that $ E $ is a cogenerator. 
Injectivity over $ R/\Ann(C) $ is also immediate, using that $ E $ is a submodule of the injective $ E(R) $ and that all maps from a module in $ \lMod{R/\Ann(C)} $ to $ E(R) $ must factor through $ E $.

(3) Notice that the sequence 
\[
\begin{tikzcd}
0 \arrow[r] & C_1 \arrow[r] & C_0 \arrow[r] & \operatorname{Im}(g) \arrow[r] & 0
\end{tikzcd}
\]
is an approximation sequence as in (\ref{approxsequence}). In particular, a module $ M $ is cogenerated by $ C $ precisely if it is annihilated by $ \Ann(C) $ and $ \Ext^1_{R/\Ann(C)}(M, C_1) = 0 $. However $ \Cogen(C) \subseteq \ker \Ext^1_R(-,C_1) \subseteq \ker \Ext^1_{R/\Ann(C)}(-,C_1) $ thus we obtain the desired identity.
\end{proof}

\begin{lem}
\label{lem:CsigmaClosedDirLim}
Let $ M $ be a module with minimal injective copresentation $ 0 \to M \to I_0 \xrightarrow{\sigma} I_1 $. 
Then 

(1) $\Ccal_\sigma = \left\{ X \in \lMod{R}\, |\,  \Ext^{1}(Y, M) = 0 \text{ for all }Y \le X\right\}$

(2) If $R$ is left artinian and $M$ is pure-injective, then $ \mathcal{C}_\sigma $ is a cosilting class.
\end{lem}

\begin{proof}
(1) The inclusion $\subseteq$ follows from the fact that  $ \mathcal{C}_\sigma$ is  closed under submodules and contained in $^{\perp_1}M $. For the reverse inclusion we refer to  the proof of  \cite[Lemma 4.17]{cosilt}.

(2) We first show that $ \mathcal{C}_\sigma $ is closed under $ \varinjlim $.
Note that $ \mathcal{C}_\sigma = {}^{\perp_1}M \cap \mathcal{C}_\rho $, where $ \rho : \operatorname{Im}(\sigma) \to I_1 $. Since $ M $ is pure-injective $ {}^{\perp_1}M $ is closed under $ \varinjlim $, thus it is enough to show that $ \mathcal{C}_\rho $ is closed under $ \varinjlim $. 
This is true for any monomorphism $ \rho $. In fact, let $ (X_i, \{ \phi_{ij} \}) $ be a directed system in $ \mathcal{C}_\rho $, and let $ f : \varinjlim X_i = X \to I_1 $. 
Then for each $ i \in I $ we obtain a commutative diagram
\[
\begin{tikzcd}
X_i \arrow[r, "\phi_i" ] \arrow[d, dashed, "h_i"] & X \arrow[d, "f" ] \\
\operatorname{Im}(\sigma) \arrow[r, "\rho"] & I_1
\end{tikzcd}
\]
Moreover, we have 
\[
(\rho h_j ) \phi_{ij} = f \phi_{j} \phi_{ij} = f \phi_i = \rho h_i 
\]
and since $ \rho $ is a monomorphism, $ h_j\phi_{ij} = h_i $, thus $ \{h_i\} $ is compatible with the directed system and it induces, by the universal properties of colimits, a factorisation $ h : X \to \operatorname{Im}(\sigma) $.
Now $ \rho h \phi_i = \rho h_i = f \phi_i $, for all $ i \in I $, thus, by uniqueness of factorisation $ \rho h = f $ and $ X \in \mathcal{C}_\rho $.

Assuming now that $ R $ left artinian, it remains to show that $ \mathcal{C}_\sigma $ is a torsionfree class. But this is immediate as it coincides with the limit closure of $ \mathcal{C}_\sigma  \cap \lmod{R} $, 
which is a torsionfree class in $ \lmod{R} $ by Lemma~\ref{cor:torsionNoetherianCat} because $ \mathcal{C}_\sigma $ is closed under submodules and extensions.
\end{proof}

We can now describe the class $A({}^{\perp_0}C)$ associated to the cosilting module $C$.

\begin{prop}
\label{thm:AlphaCotilting}
If $ C $ is a cosilting module with approximation sequence (\ref{approxsequence})  and $ \Tcal={}^{\perp_0}C $ is the associated  torsion class, then 
\[
\Alpha(\Tcal) = \Ccal_\sigma
%{}^{\perp_{1h}} C_0
\]
where $\sigma$ is a minimal injective copresentation of $C_0$.

In particular, if $R$ is left artinian, then $\Alpha(\Tcal)$ is a cosilting class.
\end{prop}

\begin{proof}
" $ \supseteq $ " : Let $ X \in \Ccal_\sigma $. We  show that for every $ T \in {}^{\perp_0} C $ and every map $ f : T \to X $ we have  $ \ker(f) \in {}^{\perp_0} C $. 

Since $ \Ccal_\sigma $ is closed under submodules, we may assume, without loss of generality, that $ f $ is an epimorphism. 
Consider the short exact sequence $ 0 \to \ker(f) \to T \to X \to 0 $. Applying $ \Hom_R(-,C_0) $ to the sequence, we obtain that $ \Hom(\ker(f), C_0) = 0 $. However, since $ C_0 $ cogenerates $ \Cogen(C) $, by Lemma~\ref{lem:approxSeqInFactor}(1), it follows that $ \Hom(\ker(f), C) = 0 $ as desired.

" $ \subseteq $ " : Let $ X \in \Alpha({}^{\perp_0}C) $. This class is closed under submodules by Lemma \ref{lem:basicAlphaBeta}, so it is enough to show that $ \Ext^1(X,C_0) = 0 $. 

Let $ 0 \to C_0 \to M \xrightarrow{f} X \to 0 $ be a short exact sequence. 
Applying the snake lemma to the commutative diagram:
\[
\begin{tikzcd}
0 \arrow[r] & F \arrow[r] \arrow[d,hook] & \operatorname{t}M \arrow[r] \arrow[d, hook] & f(\operatorname{t}M) = I \arrow[d, hook] \arrow[r] & 0 \\
0 \arrow[r] & C_0 \arrow[r] & M \arrow[r, "f"] & X \arrow[r] & 0  \\
\end{tikzcd}
\]
we obtain
\[
\begin{tikzcd}
0 \arrow[r] & F \arrow[r] \arrow[d,hook] & \operatorname{t}M \arrow[r] \arrow[d, hook] &  I \arrow[d, hook] \arrow[r] & 0 \\
0 \arrow[r] & C_0 \arrow[r] \arrow[d,two heads] & M \arrow[r, "f"] \arrow[d, two heads, "h"] & X \arrow[d, two heads] \arrow[r] & 0 \\
0 \arrow[r] & L \arrow[r] & M/\operatorname{t}M \arrow[r] & \overline{X} \arrow[r] & 0
\end{tikzcd}
\]
Since $ I $ is a submodule of $ X $ it is in $ \Alpha({}^{\perp_0}C) $, thus $ F \in {}^{\perp_0} C $. But then $ F \in \Cogen(C) \cap {}^{\perp_0} C = {0} $. 

This forces $ L = C_0 $. Then, since $ C_0 $ is split-injective in $ \Cogen(C) $ and $ M/\operatorname{t}M $ is in $ \Cogen(C) $, the third short exact sequence splits. Therefore we get a map $ g : M/\operatorname{t}M \to L $. Then the map $ g \circ h $ is a splitting epimorphism for the middle sequence.
\end{proof}

Over a noetherian ring, we again have a compatibility result thanks to the interplay between cosilting torsion pairs and  torsion pairs in $\lmod R$.
\begin{prop}
\label{prop:alphaIsLimClosure}
Let $ R $ be a left noetherian ring. If $ \Tpair{T}{F} $ is a cosilting torsion pair in $ \lMod{R} $ with restriction $ \tpair{t}{f} $, then
\[
\alpha(\mathcal{T}) = \varinjlim [ \widetilde{\alpha}(\mathbf{t}) ].
\]
\end{prop}

\begin{proof}
Observe that $\Ccal_\sigma$ is closed under coproducts by definition. By  Lemma \ref{lem:basicAlphaBeta} and Proposition~\ref{thm:AlphaCotilting}, we then have that $ \alpha(\mathcal{T}) =\Tcal\cap\,\Ccal_\sigma$ is wide and closed under coproducts, whence it is closed under direct limits. Since $ \widetilde{\alpha}(\mathbf{t}) \subseteq \alpha(\mathcal{T}) $ by Proposition \ref{prop:cosiltRestriction}, we obtain the  inclusion " $\supseteq$ ".

Conversely, if $ X \in \alpha(\mathcal{T}) $, then by Proposition \ref{prop:cosiltRestriction}, all its finitely generated submodules are in $ \widetilde{\Alpha}(\mathbf{t}) $. 
Hence, we can write $ X = \varinjlim (X_i) $ with $ X_i \in \widetilde{\Alpha}(\mathbf{t}) $. Since $ \mathcal{ F } $ is definable, the torsion radical of the torsion pair commutes with direct limits, in particular $ X = t(X) = t( \varinjlim X_i ) \cong \varinjlim  t(X_i) $. 
Now, each $ t(X_i) \in  \mathcal{T} \cap \widetilde{\Alpha}(\mathbf{t})  = \widetilde{\alpha}(\mathbf{t}) $. This proves the inclusion " $\subseteq$ ".
\end{proof}

\begin{lemdef}\cite[Corollary 3.2 and Remark 3.2]{reflective}\label{coref} A subcategory $\Ccal$ of $\lMod R$ is precovering and closed under cokernels if and only if the inclusion functor $\Ccal\hookrightarrow \lMod R$ admits a right adjoint. 

{\rm A subcategory with these properties is said to be \emph{coreflective}. We denote by $\mathbf{CWide}(R)$ the class 
%\footnote{This might be a proper class. Nevertheless, we will treat it as a set for ease of exposition.} 
of all wide coreflective subcategories of $\lMod R$.}
\end{lemdef}

\begin{thm}
\label{themapalpha}
Let $ R $ be a noetherian ring. Then the assignment
$ \Tpair{T}{F} \mapsto \alpha(\mathcal{T})$ defines  a map
\[
\alpha : \mathbf{Cosilt}(R) \longrightarrow \mathbf{CWide}(R)
\]
whose image is the class  $\overline{\mathbf{wide}}(R)$  of subcategories of $\lMod R$ which are obtained as direct limit closures of wide subcategories of $\lmod R$. Taking  the restriction  ${\rm res}: \Xcal\mapsto \Xcal\cap\lmod R$ we obtain  a commutative diagram 
$$
\xymatrix{
\mathbf{Cosilt} (R)\,\ar@{->>}[r]^{\alpha}\ar[d]_-{{\rm res}}&\overline{\mathbf{wide}}(R)\ar[d]_-{{\rm res}}\\
\mathbf{tors} (R)\ar@{->>}[r]^{\tilde{\alpha}}\ar[u]_{\varinjlim}&\mathbf{wide}(R)\ar[u]_{\varinjlim}
}
$$
In particular, for any $ \mathcal{W} \in \mathbf{ wide }(R) $, the subcategory $ \varinjlim\Wcal$ is wide and $ \varinjlim\Wcal\cap\lmod R=\Wcal$.
\end{thm}
\begin{proof}
A result
of El Bashir \cite{covers}  states that a subcategory 
of a module category is a covering class if and only if it is closed under coproducts and directed colimits and  equals the direct limit closure of some set of modules. In particular this applies to any class of the form $\varinjlim\Wcal$ for some $\Wcal\in\mathbf{wide}(R)$. Recall from Theorem~\ref{thm:wideGeneratedCogen} that $\Wcal=\widetilde{\alpha}( \mathbf{t} )$ with $ \mathbf{t} =\widetilde{\mathbf{T}}(\mathcal{W})$. It follows from Proposition~\ref{prop:alphaIsLimClosure} that  $\varinjlim\Wcal=\alpha(\Tcal)$ for $\Tcal=\varinjlim\mathbf{t}$. Hence $\varinjlim\Wcal$  is a  wide subcategory of $\lMod R$ which is covering; in particular, it is precovering and closed under cokernels, that is, coreflective in $ \lMod{R} $ by Lemma~\ref{coref}. We conclude that the map  $\alpha : \mathbf{Cosilt}(R) \longrightarrow \mathbf{CWide}(R)$ is well-defined and that  $\overline{\mathbf{wide}}(R)$ is its image.
Now apply Lemma~\ref{prop:cosiltRestriction}
to see that every $ \mathcal{W} \in \mathbf{ wide }(R) $ satisfies $ \varinjlim\Wcal\cap\lmod R=\Wcal$ and that  the diagram has the stated properties.\end{proof}

In the artinian case, we can identify when $ \alpha(\mathcal{T}) $ is also closed under products.

\begin{cor}
\label{cor:alphaBireflective}
Let $ A $ be a left artinian ring, and let  $ (\mathcal{ T } ,\Fcal)$ be a cosilting torsion pair with restriction $(\mathbf{t} , \mathbf{f})$ in  $\lmod{A} $. 
Then $ \alpha(\mathcal{ T }) $ is closed under direct products  in 
$\lMod A$ if and only if $ \widetilde{\alpha}(\mathbf{ t }) $ is functorially finite in $ \lmod{A} $. 

In other words,  a  wide subcategory $ \mathcal{W} \in \mathbf{wide}(A) $ is functorially finite if and only if $ \varinjlim{\mathcal{W}} $ is bireflective.
\end{cor}

\begin{proof}
The subcategory $ \widetilde{\alpha}(\mathbf{ t }) $ is covariantly finite if and only if $ \varinjlim \widetilde{\alpha}(\mathbf{ t }) = \alpha(\mathcal{ T }) $ is a definable subcategory, see \cite[Section 4.2]{locallyFp}. 
In this case, we can show that it is also contravariantly finite. In fact, $ \widetilde{\alpha}(\mathbf{ t }) = \lmod{B} $ for some left artinian ring $ B $ which is finitely generated as an $ A-$module.  

Indeed, assume $ \alpha(\mathcal{ T }) $ is closed under products, and therefore  a bireflective subcategory of $\lMod A$. Then, by \cite[Theorem 1.2]{quotRepFinAlg}, there exists a ring epimorphism $ A \to B $ with $ \alpha(\mathcal{T}) \cong \lMod{B} $. Consider a small progenerator of $  \alpha(\mathcal{ T }) $, which we denote again by $ B $. Then $ B $ can be written as a direct limit of objects $ B_i $ in $ \widetilde{\alpha}(\mathbf{ t }) $, in particular, it is a quotient of $ \coprod B_i $. Since $ B $ is projective in the subcategory, we have that $ B $ is actually a direct summand   of $ \coprod B_i $. But $ B $ is also compact in the category, thus it is a summand of a finite direct sum of finitely generated modules. In particular $ B $ is finitely generated. 
This shows that $ \widetilde{\alpha}(\mathbf{ t }) = \lmod{B} $.
\end{proof}

\begin{rem} \label{prop:partialWideCoprodAlpha} 
{\rm We see in Theorem~\ref{themapalpha}
that the assignment $\Wcal\mapsto T(\Wcal)$ is a right inverse of 
the map
$\alpha : \mathbf{Cosilt}(R) \longrightarrow \mathbf{CWide}(R)$. In fact,
it is shown in \cite{tau-tilt}
 that   
$\alpha(\mathbf{T}(\mathcal{W})) = \mathcal{W}$ for any wide subcategory $\Wcal$ which is closed under coproducts.}
\end{rem}

%%%%%%%%%%%%%%%%%%%%%%%%%%%%%
We now turn to a description of the class $ \Beta(\Cogen C) $.

\begin{prop}
\label{thm:BetaC1Orthogonal}
If $ C $ is a cosilting module with approximation sequence (\ref{approxsequence})  and $ \Fcal=\Cogen C$, then
\[
\Beta(\Fcal) = {}^{\perp_0} C_1
\]
In particular, $ \Beta(\Fcal) $ is a torsion class in $ \lMod{R} $. \end{prop}

\begin{proof}
"$\subseteq$" : As $ \Beta(\Cogen(C)) $ is closed under quotients ( by Lemma \ref{lem:basicAlphaBeta}), it is enough to show that we can't have a non-zero monomorphism from some object $ B \in \Beta(\Cogen(C)) $ to $ C_1 $. 

Assume that we can find a monomorphism $ i : B \to C_1 $ and consider the following pushout diagram:
\[
\begin{tikzcd}
  &  B \arrow[r, equals, ] \arrow[d, hook, "i" ] & B \arrow[d, hook] \\
0 \arrow[r] & C_1 \arrow[r] \arrow[d, two heads] & C_0 \arrow[r, "f" ] \arrow[d, two heads, "l"] & E(R) \arrow[d, equals] \\
0 \arrow[r] & \mathrm{coker}(i) \arrow[r] & P \arrow[r, "m"] \arrow[u, dashed, bend left, "r"] & E(R)
\end{tikzcd}
\]
Since $ B \in \Beta(\Cogen(C)) $, we have that $ P \in \Cogen(C) $. Thus the map $ m : P \to E(R) $ must factor through $ f $ via some $ r : P \to C_0 $. Whence we obtain that $f = m \circ l = ( f \circ r  ) \circ l $. By right minimality of $ f $, the map $ r \circ l $ is an isomorphism. Thus $ l $ is a monomorphism. This implies that $ B = 0 $.

"$\supseteq$" : Let $ X \in {}^{\perp_0}C_1 $. We must show that every map $ X \to F $ for $ F \in \Cogen(C) $ has torsionfree cokernel. Since $ {}^{\perp_0}C_1 $ is closed under quotients, without loss of generality we consider only injective maps. 

So let $ 0 \to X \to F \to M \to 0 $ be a short exact sequence. The long exact sequence obtained applying the functor $ \Hom_R(-,C_1) $ shows that $ \Ext^1_R(M, C_1) = 0 $, as $ \Ext^1_R(F, C_1) = 0 $ by Lemma \ref{lem:approxSeqInFactor}. 
Moreover, as $ M $ is a quotient of $ F $, we have $ M \in \lMod{R/\Ann(C)} $. It follows that $ M $ belongs to $ {}^{\perp_1}C_1 \cap \lMod{R/\Ann(C)} $, which coincides with 
$\Cogen(C) $  again by Lemma \ref{lem:approxSeqInFactor}.
\end{proof}

The categories  $\alpha(\Tcal)$ and $\beta(\Fcal)$ can be regarded as generalized perpendicular categories; in fact, that's what they are in the cotilting case.
\begin{rem}
\label{perpendicular}
{\rm When $C$ has injective dimension at most one, we have $\Alpha(\Tcal)={}^{\perp_{1}}C_0$ and $\alpha(\Tcal) = {}^{\perp_{0,1}}C_0$. If $C$
 is a cotilting module, we also have $\Fcal={}^{\perp_{1}}C_1$,
thus  $\beta(\Fcal) = {}^{\perp_{0,1}}C_1$.}\end{rem}

%\begin{cor}%\label{cor:betaWideCoprod} 
%If $(\Tcal,\Fcal)$ is a cosilting torsion pair, then $ \beta(\Fcal) $ is a wide coreflective subcategory. Moreover, if  $ \Fcal=\Cogen C$ for a cotilting module $C$  with approximation sequence (\ref{approxsequence}), then $\beta(\Fcal) = {}^{\perp_{0,1}}C_1$.\end{cor}

Our second main result is devoted to the map $\beta$.

\begin{thm}
\label{prop:betaInj}
Let $ A $ be a left artinian ring. Then the assignment 
        $ \Tpair{T}{F} \mapsto \beta(\mathcal{F})$   
defines an injective map
\[
\beta : \mathbf{Cosilt}(A) \longrightarrow \mathbf{CWide}(A)
\]
Taking  the restriction to ${\rm res}: \Xcal\mapsto \Xcal\cap\lmod A$ we obtain  a commutative diagram
$$
\xymatrix{
\mathbf{Cosilt} (A)\,\ar@{^{(}->}[r]^{\beta}\ar[d]_-{\rm res}&\mathbf{CWide} (A)\ar@{->>}[d]^-{\rm res}\\
\mathbf{tors} (A)\ar@{->>}[r]^{\tilde{\beta}}\;\ar[u]_{\varinjlim}&\mathbf{wide}(A)
}
$$
%where the left vertical arrow is a bijection.
%In particular, every wide subcategory of $\lmod A$ is the restriction of 
\end{thm}

\begin{rem}
{\rm
In the proof of this theorem, we will make use of the following fact: a cosilting torsion pair in the module category of a left artinian ring is uniquely determined by its torsionfree, almost torsion modules. In  \cite[Proposition 2.20]{nagoya} this  is stated for finite-dimensional algebras. However, the proof only relies on the compatibility results between small and large torsion pairs and on a result of Enomoto \cite{monobrick} giving a connection between bricks and torsionfree classes which is valid in an arbitrary abelian length category. Thus, the statement remains true in the context of left artinian rings.
}
\end{rem}

\begin{proof}
Given a cosilting class $\Fcal$, we know from Lemma \ref{lem:basicAlphaBeta} and Proposition~\ref{thm:BetaC1Orthogonal} that the class $ \beta(\Fcal) = {}^{\perp_0}C_1 \cap \Cogen(C) $
is  wide and closed under coproducts.  Moreover, it is also closed under pure quotients, as it is the intersection of two classes closed under such quotients. It follows from \cite[Theorem 2.5]{whenTheClassIsCovering} that $ \beta(\Fcal) $ is precovering. Since  $ \beta(\Fcal) $ is also closed under cokernels, we 
infer from Lemma~\ref{coref} 
 that is is coreflective. 

% \reml{artiniano serve solo qui - ma ci vuole un Remark perch\'e le ipotesi in \cite{nagoya} sono pi\`u forti}
The map $\beta$ is thus well-defined. Moreover,  if $ \beta(\mathcal{F}) = \beta(\mathcal{F}') $, then $ \mathcal{F} $ and $ \mathcal{F}' $ have the same torsionfree, almost torsion modules by Proposition \ref{prop:simpInalphabeta}. Since $A$ is  left artinian, we conclude from  \cite[Proposition 2.20]{nagoya} that $ \mathcal{F} = \mathcal{F}' $. 
 Observe further that the diagram commutes  by Lemma~\ref{prop:cosiltRestriction}
and $\tilde{\beta}$ is surjective by Theorem~\ref{thm:wideGeneratedCogen}(ii). Finally, since the left vertical arrow is a bijection,  the right vertical arrow is a surjection.
\end{proof}

In the hereditary case, we can characterize when $\beta(\Fcal)$ is closed under direct products.

\begin{cor} A cosilting torsion pair $(\Tcal,\Fcal)$ in $ \lMod{R} $ over a left artinian
hereditary ring $R$ is widely generated if and only if  $\beta(\Fcal)$ is closed under direct products 
%(and therefore  a bireflective subcategory)
 in $\lMod R$.
\end{cor}
\begin{proof}
If $\Fcal=\Wcal^{\perp_0}$ for some $\Wcal\in\mathbf{wide}(R)$, then it follows from Lemma~\ref{lem:betaCompHereditary} that $\beta(\Fcal)=\Wcal^{\perp_{0,1}}$ is closed under products.
Conversely, if $\beta(\Fcal)$ is closed under products, then by \cite[Theorem 8.1]{extPairs} there is $\Wcal\in\mathbf{wide}(R)$ such that $\beta(\Fcal)=\Wcal^{\perp_{0,1}}$, and the latter coincides with $\beta(\Wcal^{\perp_0})$ by Lemma~\ref{lem:betaCompHereditary}. Hence $\Fcal=\Wcal^{\perp_0}$ by the injectivity of $\beta$. 
\end{proof}

A cosilting class which is not widely generated will be exhibited  below. We will also see that in general  the map $\beta$ is not surjective and  the map $\tilde\beta$ is not injective.

\begin{exs}\label{betalpha}
{\rm 
 Let  $ \Lambda $ be the Kronecker algebra, i.e.~the path algebra of the quiver \begin{tikzcd}
0 \ar[r, bend left] \ar[r, bend right] & 1
\end{tikzcd} over an algebraically closed field $k$. We denote by $\mathbf{p}$, $\mathbf{t}$, and $\mathbf q$  the classes of all indecomposable  preprojective,  regular, and  preinjective modules, respectively. Recall that  $\mathbf{t}=\bigcup_{x\in\mathbb X}\mathbf{t}_x$ where $(\mathbf{t}_x)_{x\in\mathbb X}$ is a family of tubes.

 There is a complete classification of the cosilting torsion pairs in $\lMod \Lambda$. They are either generated by a finite dimensional module  $M\in \mathbf{p}\cup\mathbf{q}$, or by a set of the form $\mathbf{t}_P\cup\mathbf{q}$ determined by a subset of the tubes $\mathbf{t}_P=\bigcup_{x\in P}\mathbf{t}_x$ with $P\subseteq \mathbb X$. The torsion pair generated by $\mathbf{t}_P\cup\mathbf{q}$ with $P\ne\emptyset$ is of the form $(\Gen \mathbf{t}_P, \Fcal_P)$. In case $P=\mathbb X$ we just write $(\Gen \mathbf{t}, \Fcal)$. When $P=\emptyset$, we obtain the split  torsion pair   $(\Add \mathbf q, \Ccal)$  generated by  $\mathbf q$.  

Observe that $(\Add \mathbf q, \Ccal)$  is the only cosilting torsion pair which is not widely generated. 
In fact, $\beta(\Ccal)=  {}^{\perp_{0,1}}G$ is the perpendicular class to the generic module $ G $, and it is not closed under direct products.  For details we refer to \cite[Example 4.10 and Section 6.4]{paramTp}.

Observe further that $\alpha(\Add\mathbf{q})=\tilde\alpha(\add\mathbf{q})=0$. Indeed, if we  number the modules  $(Q_n)_{n\in\mathbb N}$ in $\mathbf{q}$  such that $\textrm{dim}_k\Hom_A (Q_{n+1},Q_{n}) = 2$, we see that every $Q_n$ is isomorphic to $Q_{n+1}/S$ for a simple regular module $S$, and therefore it can't belong to $\alpha(\Add \mathbf{q})$. This shows that $\alpha$ and $\tilde{\alpha}$ are not injective.

For a proper subset $P\subset \mathbb X$ consider now  the direct limit closure $\Wcal_P=\varinjlim\add\mathbf{t}_P$ of the wide subcategory $\add\mathbf{t}_P$ of $\lmod \Lambda$. It is a wide coreflective subcategory of $\lMod \Lambda$ by Theorem~\ref{themapalpha}, and it is not closed under direct products, because  $\add\mathbf{t}_P$ is not covariantly finite in $\lmod\Lambda$. Hence $\Wcal_P$ can't be of the form $\beta(\Fcal)$ for a widely generated torsion pair. Moreover, $\Wcal_P\ne\beta(\Ccal)$ because any simple regular $S\in\mathbf{t}_x$ with $x\in\mathbb X\setminus P$ lies in 
$\beta(\Ccal)\setminus \Wcal_P$.   Hence $\Wcal_P$ does not belong to the image of $\beta$.
This shows that the map 
$\beta:\Cosilt \Lambda\hookrightarrow\mathbf{CWide}(\Lambda)$ is not surjective.

Finally, we notice  that $(\Gen \mathbf{t},\Fcal)=({}^{\perp_0}G,\Cogen G)$ coincides with the torsion pair cogenerated by the generic module $G$.  Hence $G$ is its unique torsionfree almost torsion module, and $\beta(\Fcal)\ne 0$ while $\beta(\Fcal)\cap\lmod\Lambda=\tilde\beta(\add \mathbf p)=0$. So, the map 
${\tilde{\beta}}:\mathbf{tors}(\Lambda) \to\mathbf{wide}(\Lambda)$
is not injective. Moreover, we see that the image of $\beta$ is not contained in $\overline{\mathbf{wide}}(\Lambda)$.
%$(\Gen \mathbf{t},\Fcal)$ also yields an explicit example of a cosilting torsion pair for which $ \alpha(\mathcal{T}) $ is not closed under products, as $ \alpha(\mathcal{T}) \cap \lmod{A} =\add \mathbf{t}0$ is not  functorially finite.
}
\end{exs}

%We can now complete a large version of Theorem \ref{thm:wideGeneratedCogen}(i) for cosilting torsion pairs:

%\begin{prop}\label{prop:AlfaCoprodClsdIffCosiltWidelyGen} Let $ R $ be a ring, $ \mathcal{W} \subseteq \lMod{R} $ such that $ \mathcal{W}^{\perp_0} $ is a definable torsionfree class, then: \[\alpha(\mathbf{T}(\mathcal{W})) = \mathcal{W} \text{ if and only if } \mathcal{W} \in \mathbf{Wide}_{\coprod}(R)\] \end{prop}

%\begin{proof} If $ \mathcal{W} \in \mathbf{Wide}_{\coprod}(R) $ then $ \alpha(\mathbf{T}(\mathcal{W})) = \mathcal{W} $ by Proposition \ref{prop:partialWideCoprodAlpha}; conversely, if the equality holds, $ \mathcal{W} $ is wide and closed under coproducts by Corollary \ref{cor:alphaWideCoprod}. \end{proof}

%%%%%%%%%%%%%%%%%%%%%%%%%%%%
\section{Applications to $ \tau-$tilting infinite algebras}
%\subsection{Mutation}\label{sec:MinimalCosilting}
Throughout this section we will assume  that  $ \Lambda $ is an artin algebra.  %\reml{Per\`o citiamo risultati che sono formulati per algebre di dim. fin...}
Recall that 
there is a natural partial order on the collection of torsion classes $ \mathbf{tors}(\Lambda) $ of $ \lmod{\Lambda} $ given by inclusion. 
As shown in \cite{lattices}, the resulting poset has the structure of a complete lattice and enjoys several nice lattice-theoretic properties. 
%More explicitely, we have the following description of the meet and join of a set indexed family $ \{ \mathbf{t}_i \}_{i \in I} $ of torsion classes:\[\bigwedge_{I} \mathbf{t}_i := \bigcap_I \mathbf{t}_i \quad, \quad \bigvee \mathbf{t}_i := \widetilde{\mathbf{T}}\Big( \bigcup_I \mathbf{t}_i \Big)\]
The algebras for which this lattice is finite are called $ \tau-$\emph{tilting finite}.
A typical phenomenon in the $ \tau-$tilting infinite case is the presence of non-trivial locally maximal elements. 

\begin{de}{\rm
Given two  torsion classes  $ \mathbf{ u}$ and $\mathbf t $ in $\mathbf{tors}(\Lambda) $, we say that $\mathbf t$ \emph{covers} $\mathbf u$ if $\mathbf u\subset\mathbf t$ and there is no $\mathbf t' $ in $\mathbf{tors}(\Lambda) $ which properly contains $\mathbf u $ and is properly contained in $\mathbf t$. 

A torsion class $ \mathbf{ t } \in \mathbf{tors}(\Lambda) $ is said to be \emph{locally maximal} if there are no elements of $ \mathbf{ tors }(\Lambda) $ covering $ \mathbf{t} $. Moreover,  $  \mathbf{ t } $ is \emph{locally minimal} if there are no elements of $ \mathbf{ tors }(\Lambda) $ covered by $ \mathbf{t} $.}
\end{de}

It is shown in {\cite{minimIncl}}
 that  the torsion classes covering (respectively, covered by) $ \mathbf{t} $ are in bijection with the
 isoclasses of minimal (co)extending modules with respect to the torsion pair $\tpair{t}{f}$. Moreover, we know  from \cite{mutation} that a torsion pair covering $\mathbf{t} $ amounts to a mutation of the associated cosilting module, or more precisely, of the corresponding two-term cosilting complex. 
Rather than giving  
  further details on the concept of mutation for cosilting objects  introduced in \cite{mutation}, here we prefer to use the following equivalent characterisation from \cite[Theorem 8.8]{mutation}: 

\begin{de}[\cite{mutation}]
{\rm  Let $ \tpair{t}{f} $ and $ \tpair{u}{v} $ be  torsion pairs  in  $ \mathbf{tors}(\Lambda)$. We say that $  \tpair{t}{f} $ is a \emph{right mutation} of $ \tpair{u}{v} $, and  $ \tpair{u}{v} $  is a \emph{left mutation} of $ \tpair{t}{f} $,
if $ \mathbf{u} \subseteq \mathbf{t} $ and $ \mathbf{t} \cap \mathbf{v}\in\mathbf{wide}(\Lambda)$.}
\end{de} 

%\begin{rem}\label{rem:ICE-IKEMutation}Notice that if $ \Tpair{T}{F} $ is a right mutation of $ \Tpair{U}{V} $ then the subcategory $ \mathcal{T} \cap \mathcal{V} $ is both ICE and IKE-closed, in particular in this setting $ \mathcal{T} \subseteq \Beta(\mathcal{V}) $ and $ \mathcal{V} \subseteq \Alpha(\mathcal{T}) $, see Proposition \ref{prop:maxICE-IKE}. \end{rem}

The classes $\widetilde\alpha(\mathbf t)$ and $\widetilde\beta(\mathbf f)$ control the existence of mutations of the torsion pair $ \tpair{t}{f} $. 

%Since $ \alpha(\mathcal{T}) $ is wide, combining Proposition \ref{prop:AlphaDefinable} with Remark \ref{rem:ICE-IKEMutation} we can obtain the following
 
\begin{prop}
\label{cor:maximalMutation}
Let  $ \tpair{t}{f} $ be a torsion pair in $ \mathbf{tors}(\Lambda)$. 
%The class $\widecheck{\mathbf u} = \widetilde\Beta(\mathbf{f}) $ is the largest torsion class such that $  (\widecheck{\mathbf u},\widecheck{\mathbf v})  $ is a right mutation of $ \tpair{t}{f} $. 
%(3)
The torsion class  $\mathbf{t}$ is locally minimal if and only if $\tilde\alpha(\mathbf{t})=0$, if and only if $ \tpair{t}{f} $  admits no proper left mutation. 
Moreover,  $\mathbf{t}$   is locally maximal if and only if $\tilde\beta(\mathbf{f})=0$, if and only if $ \tpair{t}{f} $  admits no proper right mutation. 
\end{prop}
\begin{proof}
We know from {\cite[Theorem 2.3.2]{minimIncl}}
that  $\mathbf{t}$ is locally minimal if and only if there are no (finitely generated) torsion, almost torsionfree modules. By Proposition~\ref{prop:simpInalphabeta} this means that $\widetilde\alpha(\mathbf t)=0$. Moreover, the latter is equivalent to $\mathbf f=\widetilde\Alpha(\mathbf{t}) $ by Lemma~\ref{lem:basicAlphaBeta}. This amounts to saying that there is  no proper left mutation. Indeed, it is shown in {\cite[Corollary 9.9]{mutation}}
that  $ \widehat{\mathbf v} = \widetilde\Alpha(\mathbf{t}) $ is the largest torsionfree class such that 
$  (\widehat{\mathbf u},\widehat{\mathbf v})$ is a left mutation of $ \tpair{t}{f} $. 

The second statement  is proven dually, since we know from {\cite[Theorem 1.0.2]{minimIncl}} that  $\mathbf{t}$ is locally maximal if and only if there are no   finitely generated torsionfree, almost torsion modules. \end{proof}

%%%%%

Our next aim is to exhibit a condition on a torsion pair which ensures the  existence of a mutation.
The following concept is introduced in \cite[Definition 4.12, Remark 4.18]{paramTp}.

\begin{de}{\rm
A cosilting module $ C $ over a ring $R$ with cosilting class $\Fcal=\Cogen(C)$ and approximation sequence (\ref{approxsequence})
is \emph{minimal} if 
\begin{itemize}
\item[(i)] $ \beta(\Fcal) 
%= \Cogen(C) \cap {}^{\perp_0}C_1 
$ is closed under direct products (and thus a bireflective subcategory) in $ 	\lMod{R} $, 
\item[(ii)] $ \Hom_R(C_0, C_1) = 0. $
\end{itemize}}
\end{de}

The interest in minimal cosilting modules stems from their connection with ring epimorphisms. For example, over a commutative noetherian or over a hereditary ring, minimal cosilting modules are in one-one-correspondence with  homological ring epimorphisms, up to equivalence. For details we refer to \cite{paramTp}.

In general, it is not easy to understand if a certain cosilting class is cogenerated by a minimal cosilting module, however over a hereditary algebra there is a handy criterion.

\begin{prop}
\label{prop:MinCosOverHeredIsWide}
Let $ A $ be a left artinian hereditary ring. Then a cosilting module $ C $ is equivalent to a minimal one if and only if the cosilting torsion pair cogenerated by $C$ is widely generated.
% there exists a wide subcategory $ \mathcal{W} \subseteq \lmod{\Lambda} $ such that $ \Cogen(C) = \mathcal{W}^{\perp_0} $.
\end{prop}

\begin{proof}
If $ C $ is a minimal cosilting module, then $\beta(\Cogen C)$ is bireflective, hence
%it arises from a homological ring epimorphism. B
by \cite[Theorem 6.1]{extPairs} it is of the form $ \mathcal{W}^{\perp_{0,1}}$ for some $\Wcal\in\mathbf{wide}(A)$. By Lemma \ref{lem:betaCompHereditary} we infer that $\beta(\Cogen C)=\beta(\mathcal{W}^{\perp_0}) $,
%every such epimorphism is a universal localisation. The corresponding bireflective subcategory is $ \beta(\Cogen(C)) $. By \cite[Theorem 2.3]{uniLocHered}, every universal localisation is obtained as the perpendicular category of a unique wide subcategory $ \mathcal{W}$ of $ \lmod{\Lambda} $, and we claim that $ {}^{\perp_0}C = \mathbf{T}(\mathcal{W}) $.This 
and the claim follows
 immediately from the injectivity of $ \beta $ in Theorem \ref{prop:betaInj}.
 %, as $ \beta(\mathcal{W}^{\perp_0}) = \mathcal{W}^{\perp_{0,1}} = \beta(\Cogen(C)) $, 

Conversely, every torsion pair $ ( \mathbf{T}(\mathcal{W}) , \mathcal{W}^{\perp_0} ) $  with $\Wcal\in\mathbf{wide}(A)$ is cosilting, and 
%The torsionfree class is  definable, therefore, this is a cosilting torsion pair. Moreover, 
$ \beta(\mathcal{W}^{\perp_0}) = \mathcal{W}^{\perp_{0,1}}  $ is a bireflective subcategory.  
Now \cite[Theorem 4.16 and Corollary 4.18]{paramTp} tell us that  in the hereditary case the map $\beta$ restricts to a bijection between minimal cosilting modules and bireflective subcategories. So there 
is a minimal cosilting module $C$ such that  $\beta(\Cogen C)=\beta(\mathcal{W}^{\perp_0}) $
and the claim follows applying  Theorem \ref{prop:betaInj} once again.
%, we conclude that this new torsion pair is the original one, and thus it is cogenerated by a minimal cosilting module as required.
\end{proof}

%\begin{ex}Let $ \Lambda = kQ $ be a representation-infinite (hereditary) algebra. Then we always have an easy example of a non-minimal cosilting torsion pair: let $ \mathbf{q} $ be a collection of iso-classes of all the preinjective modules in $ \lmod{\Lambda} $, then $ ( \mathbf{T}(\mathbf{q}),  \mathbf{q}^{\perp_0} ) $ is a cosilting torsion pair in $ \lMod{\Lambda} $ which is not widely generated, thus not minimal.\end{ex}

%\subsection{Torsion, almost torsionfree modules for minimal cosilting modules}

We will prove that, over an artin algebra, all minimal cosilting modules (with the exception of injective cogenerators) admit some torsion, almost torsionfree module. The following preliminary result holds true over an arbitrary ring.

\begin{lem}
Let  $ C $ be a  cosilting module over a ring $R$
with approximation sequence (\ref{approxsequence}), and let $ (\Tcal,\Fcal)=({}^{\perp_0}C, \Cogen C) $ be the associated  torsion pair.
\begin{enumerate}
\item If $ \alpha(\Tcal) = 0 $, then $ C_0 $ is cosilting.
\item If $ C $ is cotilting, then $ C_0 $ is cotilting if and only if $ \alpha(\Tcal) = 0 $.
\item If $ C $ is a minimal cosilting module and the module $ C_0 $ is cosilting, then $ C_1 = 0 $.
\end{enumerate}
\end{lem}

\begin{proof}
(1)  Let $ \sigma $ be the minimal injective copresentation of $ C_0 $. We know from Proposition~\ref{thm:AlphaCotilting} that $ \mathcal{C}_\sigma  = \Alpha(\Tcal) $, and by Lemma~\ref{lem:basicAlphaBeta} the latter class equals  $ \Fcal $, as $ \alpha( \Tcal ) = 0 $  by assumption.
%We must show that there is some injective copresentation $ \omega $ of $ C_0 $, such that $ \mathcal{C}_{\omega} = \Cogen(C_0) = \Cogen(C) $.  Notice that, for every copresentation $ \omega $, the class $ \mathcal{C}_\omega $ is closed under submodules. Moreover, $ \mathcal{C}_\omega \subseteq {}^{\perp_1} C_0 $ : if $ M \in \mathcal{C}_{\omega} $, then $ M $ is also in $ \mathcal{C}_{\pi} $ where $ \pi : I_0 \to \operatorname{Im}(\omega) $, and applying the functor $ \Hom(M,-) $  to the short exact sequence\[\begin{tikzcd}0 \arrow[r] & C_0 \ar[r] & I_0 \arrow[r, "\pi"] & \operatorname{Im}(\omega) \arrow[r] & 0\end{tikzcd}\]we obtain that $ \Ext^1(M, C_0) = 0 $.Moreover, by Lemma~\ref{lem:approxSeqInFactor}, the module $ C_0 $ is a summand of a cosilting module equivalent to $ C $,  thus  by \cite[Lemma 4.13]{cosilt}, $ \Cogen(C) \subseteq \mathcal{C}_\omega $, with $ \omega $ the minimal injective copresentation of $ C_0 $. Hence $ \Cogen(C_0) \subseteq \mathcal{C}_{\omega} \subseteq \Cogen(C_0) $. 
Since $ \Fcal=\Cogen(C_0) $, we conclude that $ C_0 $ is cosilting with respect to $ \sigma $.
 
 (2) We know from Remark~\ref{perpendicular} that $ \alpha(\Tcal) = {}^{\perp_{0,1}}C_0 $. Thus $ \alpha({}\Tcal) = 0 $ if and only if $ \Cogen(C_0) = {}^{\perp_1}C_0 $, that is, if and only if $ C_0 $ is cotilting.
 
 (3) By Lemma~\ref{lem:approxSeqInFactor} we have that $ \Cogen(C_0) = \Cogen(C) $, thus the two cosilting modules are equivalent and $ \Prod(C) = \Prod(C_0) $. Therefore, $ C_1 \in \Prod(C_0) $. However, by assumption, $ \Prod(C_0) \subseteq \beta(\Cogen(C)) $. But by Proposition~\ref{thm:BetaC1Orthogonal} we have $ \beta(\Cogen(C)) \subseteq {}^{\perp_0}C_1 $, thus $ C_1 = 0 $.
\end{proof}

%\begin{rem}The dual case, with the condition $ \beta(\Fcal) = 0 $, is not interesting: by the injectivity of $ \beta $ this assumption forces $ C_0 = C_1 = 0 $ and trivially they are both cosilting modules.\end{rem}

\begin{prop}
\label{prop:MinimalCosHaveTTF}
Let $ \Lambda $ be an artin algebra, and let $ C $ be a minimal cosilting module with associated torsion pair $(\Tcal, \Fcal)$. Assume that $ \Fcal \ne \lMod{\Lambda} $. 
Then $ \alpha(\Tcal) \ne 0 $, and the torsion pair has some torsion, almost torsionfree module. 
\end{prop}

\begin{proof}
Consider again the approximation sequence (\ref{approxsequence}).
If $ \alpha(\Tcal) = 0 $, then $ C_0 $ is a cosilting module and thus $ C_1 = 0 $. This implies that $ C_0 $ is a finitely generated cosilting module, or in other words, $ C_0 $ is support $ \tau^{-1}-$tilting. 
Moreover, the (functorially finite) torsion class $ \mathbf{t}=\Tcal\cap \lmod{\Lambda}$ corresponding to the  torsionfree class cogenerated by $ C_0 $ in $ \lmod{\Lambda} $ must satisfy $\widetilde{\alpha}(\mathbf{t})=0$ and therefore be locally minimal by Proposition~\ref{cor:maximalMutation}. But if $ \mathbf{t}\ne 0$, then by \cite[Theorem 3.1]{g-vectors}  it is possible to find, by means of mutation, a torsion class $\mathbf{u} $ which is covered  by $ \mathbf{t}$.  Thus, we must have $ \lmod{\Lambda} = \operatorname{cogen}(C_0) $, which contradicts our hypothesis by Theorem \ref{thm:CB-bij}.

Hence $ \alpha(\Tcal) \ne 0 $, which by Proposition \ref{prop:alphaIsLimClosure} amounts to $ \widetilde{\alpha}(\mathbf t ) \ne 0 $. Now $ \widetilde{\alpha}(\mathbf t )  $ has some simple object, and
this is a torsion, almost torsionfree module for  $ (\Tcal, \Fcal)$ by Proposition~\ref{prop:simpInalphabeta}.\end{proof}
In other words,  minimal cosilting modules always admit left mutation.
\begin{cor}\label{locmin}
Let $\Lambda$ be an Artin algebra and $(\mathbf{t},\mathbf{f})$ be a non-trivial torsion pair in $\lmod \Lambda$. If the  associated cosilting torsion pair $(\Tcal,\Fcal)=(\varinjlim\mathbf{ t },\varinjlim\mathbf{ t })$ in $\lMod\Lambda$ is cogenerated by a minimal cosilting module, then $\mathbf{t}$ is not locally minimal. 
\end{cor}

It is proved in \cite{nagoya} that over any $ \tau-$tilting infinite  artin algebra there exists a torsion class in $\mathbf{tors}(\Lambda)$ which is  locally maximal and not functorially finite, and dually, there exists one which is locally minimal and not functorially finite. From the discussion above we can see that this ``pathological'' behaviour of $\mathbf{tors}(\Lambda)$ is directly connected with pathological behaviour of the corresponding cosilting modules:  it ensures both the existence of large torsionfree, almost torsion modules (cf.~\cite[Lemma 3.13]{nagoya}) and of large non-minimal cosilting modules.
Let us  collect our findings in a number of new characterizations of  $ \tau-$tilting finite  algebras.
%From the proof of the Theorem above and of Theorem \ref{thm:brickAuslander}, The wide subcategory closed under coproducts $ \overline{\mathcal{W}} $ extending a given wide subcategory $ \mathcal{W} $ of $ \lmod{\Lambda} $ given in Corollary \ref{cor:sharpWideCorollary} is not unique in general. Recall the following result:\begin{lem}[{\cite[Corollary 3.11]{wideLoc}}]\label{lem:funfinTpairHasFunFinAlpha}Let $ \Lambda $ be an artin algebra and $ \mathbf{t} $ be a functorially finite torsion class in $ \lmod{\Lambda} $, then $ \widetilde{\alpha}(\mathbf{t}) $ is a functorially finite wide subcategory.  If $ \Lambda $ is $ \tau-$tilting finite every wide subcategory of $\lmod{\Lambda}$ is functorially finite.\end{lem}\begin{rem}Given a functorially finite wide subcategory $ \mathcal{W} $ of $ \lmod{\Lambda} $, with $ \Lambda $ an arbitrary artin algebra, it is not always true that $ \widetilde{\mathbf{T}}(\mathcal{W}) $ is functorially finite, see \cite[Example 3.13]{semibricks}.\end{rem} We present a further categorical characterisation of $ \tau$-tilting finiteness:

\begin{thm}
\label{prop:tauFiniteWideCop}
The following statements are equivalent for an artin algebra $ \Lambda $.
\begin{itemize}
\item[(i)] $ \Lambda $ is $ \tau$-tilting finite.
\item[(ii)] Every cosilting module which is not equivalent to a finitely generated one is minimal.
\item[(iii)] Every wide subcategory  of $ \lMod{\Lambda} $ closed under coproducts
is  the direct limit closure of a wide subcategory of $\lmod\Lambda $.
\item[(iv)]
For every $ \mathcal{W} \in\mathbf{wide}(\Lambda) $ there exists a unique wide subcategory $ \Xcal $ of $ \lMod{\Lambda} $ closed under coproducts such that $ \mathcal{W} =  \Xcal \cap \lmod{\Lambda} $. 
\item[(v)]  If $ \Xcal $ is a wide subcategory of $ \lMod{\Lambda} $ closed under coproducts, then $ \mathcal{X} \cap \,\lmod{\Lambda} = 0 $ if and only if $ \mathcal{X} = 0 $.
\item[(vi)] The class of wide subcategories of $ \lMod{\Lambda} $ closed under coproducts is a finite set.
\end{itemize}
\end{thm}

\begin{proof}
First of all, recall from \cite[Theorem 4.8]{tau-tilt} that over a  $ \tau$-tilting finite algebra, all torsion(free) classes in $\lMod\Lambda$ are given by finitely generated (co)silting modules.

(i) $\Rightarrow $ (ii) is then  trivial. 

(ii) $\Rightarrow $ (i): It follows from Corollary~\ref{locmin} and Remark~\ref{ff} that there cannot exist a torsion class in  $\mathbf{tors}(\Lambda)$ which is  both locally minimal and not functorially finite. Thus $\Lambda $ is $ \tau$-tilting finite by the dual version of \cite[Corollary 3.10]{nagoya}.

(ii) $ \Rightarrow $ (iii): If $ \Xcal $ is a  wide subcategory of $\lMod\Lambda$ which is closed under coproducts, then $\Xcal=\alpha\mathbf{T}( \Xcal)$ by Remark~\ref{prop:partialWideCoprodAlpha}. Moreover, since $ \Lambda $ is $ \tau$-tilting finite,  the torsion pair  $ (\mathbf{T}( \Xcal), \Xcal^{\perp_0}) $  is a cosilting torsion pair, so $\Xcal\in\overline{\mathbf{wide}}(\Lambda)$ by Theorem~\ref{themapalpha}. 

%For the uniqueness, we pick a wide subcategory  $ \Xcal $ of $\lMod\Lambda$ which is closed under coproducts and restricts to $ \mathcal{W} $. Then $\Xcal=\alpha\mathbf{T}( \Xcal)$ by Remark~\ref{prop:partialWideCoprodAlpha}. Moreover, since $ \Lambda $ is $ \tau$-tilting finite,  the torsion pair  $ (\mathbf{T}( \Xcal), \Xcal^{\perp_0}) $  is a cosilting torsion pair, and therefore $\Xcal\in\overline{\mathbf{wide}}(\Lambda)$.  As the map $\varinjlim:{\mathbf{wide}}(\Lambda)\rightarrow\overline{\mathbf{wide}}(\Lambda)$ and the restriction to $\lmod \Lambda$ are mutually inverse  bijections, we conclude that   $\Xcal=\varinjlim(\Xcal \cap \lmod{\Lambda}) =\varinjlim\Wcal$.

(iii) $ \Rightarrow $ (iv):  By Theorem~\ref{themapalpha},  restriction to $\lmod \Lambda$ induces a  bijection between the
wide subcategories  of $ \lMod{\Lambda} $ closed under coproducts  and ${\mathbf{wide}}(\Lambda)$.

(iv) $ \Rightarrow $ (v) is immediate.

(v) $ \Rightarrow $ (i): We apply Theorem~\ref{prop:betaInj}. 
Given a torsion pair  $\tpair{t}{f}$ in $\mathbf{tors}(\Lambda)$, we know that $\beta$ maps  $\Fcal= \varinjlim \mathbf{f}$ to a wide subcategory closed under coproducts which restricts to $\tilde\beta(\mathbf{f})$.
Our assumption  then tells us  that $\tilde\beta(\mathbf{f}) = 0 $ implies $ \beta(\Fcal) = 0 $, hence $\Fcal=0$ by the injectivity of $\beta$. From Proposition~\ref{cor:maximalMutation} and \cite[Corollary 3.10]{nagoya} 
we deduce that
 $ \Lambda $ is $ \tau$-tilting finite.
 
(i) $\Rightarrow$ (vi): When $ \Lambda $ is $ \tau$-tilting finite, the map $\tilde\alpha$ induces a bijection between $\mathbf{tors}(\Lambda)$ 
and $\mathbf{wide}(\Lambda)$, as proved in  \cite[Corollary 3.11]{wideTorsion}. The statement then follows from condition (iii).

(vi) $\Rightarrow$ (i): By the injectivity of the map $\beta: \mathbf{Cosilt}(\Lambda) \longrightarrow \mathbf{CWide}(\Lambda)$ in Theorem~\ref{prop:betaInj} we see that (vi)  implies finiteness of $ \mathbf{Cosilt}(\Lambda)$. Now use Theorem \ref{thm:CB-bij} to conclude.
 \end{proof}

%Finally, we turn to the additional statement. We have already seen that condition (ii) implies that every wide subcategory  of $ \lMod{\Lambda} $ closed under coproducts is of the form $\varinjlim\Wcal$ for some $\Wcal\in{\mathbf{wide}}(\Lambda)$. Moreover, we know by
In \cite[Corollary 3.11]{wideTorsion} it is also shown  that over a $ \tau-$tilting finite algebra every wide subcategory of $\lmod\Lambda$ is functorially finite. 
As observed in Corollary~\ref{cor:alphaBireflective} and its proof,  this means that every category $\Xcal$ in $\overline{\mathbf{wide}}(\Lambda)$ is closed under products, and in fact  there even exists a ring epimorphism $ \Lambda \to \Gamma $ to an artin algebra $\Gamma$ such that $ \Xcal \cong \lMod{\Gamma} $.
We can then restate the equivalence of (i) and (iii) in Theorem~\ref{prop:tauFiniteWideCop} as follows: 

\begin{cor}\label{modulecat}
An artin algebra $ \Lambda $ is $ \tau-$tilting finite if and only if every wide subcategory  of $ \lMod{\Lambda} $ closed under coproducts is equivalent to the category of modules over some artin algebra.
\end{cor}

We close this section with some open questions.

\begin{ques} {\rm The following is a list of necessary conditions which are satisfied
when $ \Lambda $ is a $ \tau-$tilting finite  artin algebra (see the discussion above and \cite[Theorem 4.8]{tau-tilt}). 
Is any of them a sufficient condition?
\begin{enumerate}
\item Every wide subcategory of $ \lmod{\Lambda} $ is functorially finite.
\item Every wide subcategory closed under coproducts of $ \lMod{\Lambda} $ is closed under products.
\item The target of any ring epimorphism $ \Lambda \to \Gamma $ with $\Tor_1^\Lambda(\Gamma,\Gamma)=0$ is an artin algebra.
\end{enumerate}
 Note that  (2) implies (1). Moreover, (2) and (3) imply that  $ \Lambda $ is $ \tau-$tilting finite by Corollary~\ref{modulecat}. }\end{ques}

%%%%%%%%%%%%%%%%%%%%%%%

\section{Torsion pairs and Ext-orthogonal pairs}

In this section we give some applications to Ext-orthogonal pairs over a hereditary ring.

\begin{de}[{\cite[Def. 2.1 ]{extPairs}}]
{\rm Let $ R $ be a ring. A pair $ \Tpair{X}{Y} $  of full subcategories of $ \lMod{R} $ is said to be an \emph{Ext-orthogonal pair} if:
\begin{align*}
X \in \mathcal{X} & \iff \forall n \in \mathbb{Z}\  \Ext^{n}(X,\mathcal{Y}) = 0 \\
Y \in \mathcal{Y} & \iff \forall n \in \mathbb{Z}\  \Ext^{n}(\mathcal{X}, Y) = 0 \\
\end{align*}
An Ext-orthogonal pair is \emph{complete} if for all $ M \in \lMod{R} $ we have an exact sequence:
\[
\begin{tikzcd}
0 \arrow[r] & Y_M \arrow[r] & X_M \arrow[r] & M \arrow[r] & Y^M \arrow[r] & X^M \arrow[r] & 0
\end{tikzcd}
\]
with $ X_M, X^M \in \mathcal{ X } $ and $ Y_M, Y^M \in \mathcal{Y} $.}
\end{de}

%The following two results will be used several times in the rest of the section:

%\begin{prop}[{\cite[Proposition 3.1]{extPairs}}]\label{prop:homoEpiExtPairs}

%Let $ R $ be an hereditary ring, $ f : R \to S $ a homological ring epimorphism. Let $ \mathcal{ Y } = f^*(\lMod{S}) $ and set $ \mathcal{X} = {}^{\perp_{0,1}} \mathcal{Y} $, $ \mathcal{Z} = \mathcal{Y}^{\perp_{0,1}} $. Then $ \Tpair{X}{Y} $ and $ \Tpair{Y}{Z} $ are complete Ext-orthogonal pairs in $ \lMod{R} $ with $ \mathcal{Y} = (\ker f \oplus \operatorname{coker} f)^{\perp_{0,1}} $ and $ \mathcal{Z} = S^{\perp_{0,1}} $.\end{prop}

%\begin{prop}[{\cite[Theorem 5.1]{extPairs}}]\label{prop:finTypeExtPairs}
%Let $ R $ be an hereditary ring and $ \Tpair{X}{Y} $ an Ext-orthogonal pair in $ \lMod{R} $. The following statements are equivalent:\begin{itemize}\item[(i)] $ \mathcal{Y} $ is closed under coproducts.\item[(ii)] $ \mathcal{X} = \varinjlim ( \mathcal{X} \cap \lmod{R}) $\item[(iii)] There exists a subcategory $ \mathcal{C} \subseteq \lmod{R} $ such that $ \mathcal{C}^{\perp_{0,1}} = \mathcal{Y} $.\end{itemize}\end{prop}

%\subsubsection{Minimal cosilting modules and Ext-orthogonal pairs}

As noticed in \cite{extPairs}, every complete Ext-orthogonal pair over a hereditary ring gives rise to a torsion pair (and a cotorsion pair) from which it can be recovered. 
\begin{prop}
\label{prop:ExtOrthEmbedding}
Let $ A $ be a hereditary ring, $ \Tpair{X}{Y} $ a complete Ext-orthogonal pair. Then there is a (uniquely determined) torsion pair $ \Tpair{T}{F} $ in $ \lMod{A} $ such that $ \Tpair{X}{Y} = ( \alpha(\mathcal{T}), \beta(\mathcal{F})  ) $. 
\end{prop}

\begin{proof}
As in \cite{extPairs}, we consider the torsion pair  $ ( \mathbf{T}(\mathcal{X}), \mathcal{X}^{\perp_0} ) $  generated by $ \mathcal{X} $. 

We have that $ \Cogen(\mathcal{Y}) \subseteq \mathcal{X}^{\perp_0} $, since the latter is a torsionfree class containing $ \mathcal{Y} $. 
Moreover, if $ L \in \mathcal{X}^{\perp_0} $, the approximation sequence of the Ext-orthogonal pair gives an embedding $ L \to Y^{L} $ with $ Y^{L} \in \mathcal{Y} $. 
Thus $ \Cogen(\mathcal{Y}) = \mathcal{X}^{\perp_0} $. In a similar way we can obtain that $ \mathbf{T}(\mathcal{X}) = \Gen(\mathcal{X}) $.

Now, being the left part of an Ext-orthogonal pair, $ \mathcal{ X } $ is a wide subcategory closed under coproducts. Thus, Remark \ref{prop:partialWideCoprodAlpha} gives $ \alpha( \mathbf{T}(\mathcal{X}) ) = \mathcal{ X } $.

Moreover, $ \mathcal{Y} $ is a wide subcategory closed under products. We prove that $\beta(\Cogen(\mathcal{Y})) = \mathcal{Y} $. 

 "$\subseteq$" : If $ B \in \beta(\Cogen(\mathcal{Y})) $,  there is some element $ Y \in \mathcal{Y} $ and a short exact sequence $ 0 \to B \to Y \to F \to 0 $ with $ F \in \Cogen(\mathcal{Y}) $. In particular, as $ F $  can also be embedded in some $ Y' \in\Ycal$, the module $ B $ can be realized as the kernel of a map in $ \mathcal{Y}$. 

"$\supseteq$" : Let $ Y \in \mathcal{Y} $ and $ Y \to F \to M \to 0 $ a short exact sequence with $ F \in \Cogen(\mathcal{Y}) $. Once again, we can embed $ F $ in some $ Y' \in \mathcal{Y} $. The cokernel $ C $ of the composite $ Y \to F \to Y' $ is then a module in $ \mathcal{Y} $.
Applying the snake lemma to the diagram
\[
\begin{tikzcd}
& Y \arrow[r, "f"] \arrow[d, two heads] & F \arrow[d, hook, "i"] \arrow[r] & M \arrow[d] \arrow[r] & 0 \\
0 \arrow[r] & \operatorname{Im}(i \circ f)  \arrow[r] & Y' \arrow[r] & C \arrow[r] & 0 
\end{tikzcd}
\]
we can see that $ M $ embeds in $ C $. Thus $ M \in \Cogen(\mathcal{Y}) $ and $ Y \in \beta(\Cogen(\mathcal{Y})) $.

For uniqueness, let $ (\mathcal{T}, \mathcal{F}) $ be a torsion pair with $ \Tpair{X}{Y} = ( \alpha(\mathcal{T}), \beta(\mathcal{F})  ) $. Then, obviously $ \mathbf{T}(\mathcal{X}) \subseteq \mathcal{T} $ and $ \Cogen(\mathcal{Y}) \subseteq \mathcal{F} $. But the inclusion of the torsion classes is equivalent to the reverse containment for the torsionfree classes, thus the two torsion pairs must coincide.
\end{proof}

\begin{ex} {\rm Not every torsion pair gives rise to an Ext-orthogonal pair. As an example, recall from Example~\ref{betalpha} that the cosilting torsion pair $(\Add\mathbf{q},\Ccal)$ over the Kronecker algebra satisfies $(\alpha(\Add\mathbf{q}),\beta(\Ccal))=(0,{}^{\perp_{0,1}}G)$. }
\end{ex} 

We can characterize the torsion pairs associated to certain Ext-orthogonal pairs with distinguished properties.

\begin{cor}
\label{cor:CosExt}
Let $ A $ be a  left artinian hereditary ring, and  $ \Tpair{X}{Y} $ a complete Ext-orthogonal pair. Then the corresponding torsion pair is cogenerated by a (minimal) cosilting module if and only if 
%$ \mathcal{X} = \varinjlim ( \mathcal{X} \cap \lmod{A} ) $.All the cosilting torsion pairs obtained in this way are minimal, therefore 
$ \mathcal{Y} $ is a bireflective subcategory of $ \lMod{A} $.
\end{cor}

\begin{proof}
$\Ycal$ is bireflective if and only if it is of the form $\Ycal=\Wcal^{\perp_{0,1}}$ for some $\Wcal\in\mathbf{wide}(A)$. But then the associated torsion pair must coincide with $(T(\Wcal),\Wcal^{\perp_0})$ by uniqueness and Lemma~\ref{lem:betaCompHereditary}, and it is therefore cogenerated by a minimal cosilting module by Proposition~\ref{prop:MinCosOverHeredIsWide}.

Conversely, if the torsion pair is cosilting,  then $ \mathcal{X} = \varinjlim ( \mathcal{X} \cap \lmod{A}) $ by Proposition \ref{prop:alphaIsLimClosure}. By \cite[Theorem 5.1]{extPairs}, this means that $ \mathcal{Y} $ is closed under coproducts and is thus bireflective.
\end{proof}

%%%

There is a  dual concept for minimal cosilting modules. Minimal silting modules are defined for general rings, but here we will use the following, more accessible, definition:

\begin{de}[{\cite[Definition 5.4]{siltEpi} }]{\rm
Let $ A $ be a hereditary ring. A silting $ A-$module $ T $ is \emph{minimal silting} if $ A $ admits an $ \Add(T)$-envelope.}
\end{de}
Recall that for a  cosilting torsion pair $(\Tcal,\Fcal)$ with approximation sequence (\ref{approxsequence}), we have $\alpha(\Tcal)={}^{\perp_{0,1}}C_0$ when $A$  is hereditary, cf. Remark~\ref{perpendicular}. Dually, we can consider the silting torsion pair $(\Gen(T), T^{\perp_{0}})$ and show the following.

\begin{lem}
\label{lem:betaOrthMinSilt}
Let $ T $ be a minimal silting module over a hereditary ring $ A $. Let $ A \to T_0 \to T_1 \to 0 $ be the exact sequence induced by the $ \Add(T)$-envelope. 
Then $ \beta(T^{\perp_0}) = T_0\,^{\perp_{0,1}} $. 
\end{lem}

It is shown in \cite[Theorem 5.8]{siltEpi} that minimal silting modules over  hereditary rings are in one-one-correspondence with homological ring epimorphisms via the map $\alpha$. More precisely, given
 a minimal silting module $ T $ over a hereditary ring $A$, the wide subcategory $ \alpha(\Gen(T)) $ is bireflective and thus there is a ring epimorphism $ \lambda : A \to B $ such that $ \lambda^*(\lMod{B}) = \alpha(\Gen(T)) $.  Then $\Gen (T)=\Gen (_AB)$, 
and the induced $ A-$module map $ A \to {}_AB $ is a $ \Gen(B)-$envelope,
 thus $\beta(T^{\perp_0})= B^{\perp_{0,1}} $.

\begin{cor}
Let $ A $ be a hereditary ring,
%Then there is a bijection between minimal silting torsion pairs and (complete) Ext-orthogonal pairs $ \Tpair{X}{Y} $ with $ \mathcal{X} $ bireflective.
 and  $ \Tpair{X}{Y} $ a complete Ext-orthogonal pair. Then the corresponding torsion pair is generated by a minimal silting module if and only if 
%$ \mathcal{X} = \varinjlim ( \mathcal{X} \cap \lmod{A} ) $.All the cosilting torsion pairs obtained in this way are minimal, therefore 
$ \mathcal{X} $ is a bireflective subcategory of $ \lMod{A} $.
\end{cor}
\begin{proof}
By the  observations above, we can assign to each minimal silting module $T$ a complete Ext-orthogonal pair $  (\alpha(\Gen(T)), \beta(T^{\perp_0})) = ( \lambda^*(\lMod{B}), B^{\perp_{0,1}}  ) $  with the required property, see \cite[Proposition 3.1]{extPairs}.
Conversely, 
if $ \Tpair{X}{Y} $ is any complete Ext-orthogonal pair with $ \mathcal{X} $ bireflective, then $ \mathcal{X} = \lambda^*(\lMod{B}) $ for some ring epimorphism.
Thus, once again by uniqueness and \cite[Proposition 3.1]{extPairs}, this pair is obtained as above from the minimal silting torsion pair $ (\Gen(B), B^{\perp_0} ) $.
\end{proof}
%%%%%%%%%%%

We can now combine these observations with the fact that the homological ring epimorphisms starting at a hereditary ring $A$ are precisely the universal localizations $A\to A_\Wcal$ of $A$ at wide subcategories $\Wcal\in\mathbf{wide}(A)$, see \cite[Theorem 6.1]{extPairs} for details. The following result is a variation of
\cite[Theorem 8.1]{extPairs}.

\begin{thm} 
\label{wide and extorth}
If $ A$ is a hereditary ring, there are bijections between
\begin{enumerate}
\item wide subcategories of $\lmod A$;
\item (complete) Ext-orthogonal pairs $ \Tpair{X}{Y} $ with $ \mathcal{Y} $  bireflective.
\item (complete) Ext-orthogonal pairs $ \Tpair{X}{Y} $ with $ \mathcal{X} $  bireflective.
\item  minimal cosilting torsion pairs;
\item minimal silting torsion pairs.
\end{enumerate}
\end{thm}
\begin{proof}
The bijections are given as follows.
$$\begin{array}{ccccl}
\text{Bijection} &&&& \text{Assignment}\\
\hline
 (1)\rightarrow (2) &&&& \Wcal\mapsto\; (\varinjlim\Wcal,\Wcal^{\perp_{0,1}})\\
 (1)\rightarrow (3) &&&& \Wcal\mapsto\; (\Wcal^{\perp_{0,1}}, A_\Wcal)\\
 (1)\rightarrow (4) &&&&  \Wcal\mapsto\; (T(\Wcal),\Wcal^{\perp_{0}})\\
 (1)\rightarrow (5) &&&& \Wcal\mapsto\; (\Gen A_\Wcal=\Wcal^{\perp_{1}}, A_\Wcal\,^{\perp_{0,1}})
 \end{array}$$
\end{proof}

%%%%%%%%%%%%%%%%%%%%%%%%%%%%%%%%%%%%%%%

\section{Wide subcategories arising from pure-injectives}\label{tame}
% over tame hereditary algebras}

In Theorem~\ref{wide and extorth} we have described the Ext-orthogonal pairs of the form $(\Xcal,\Ycal)=(\alpha(\Tcal), \beta(\Fcal))$ for some  widely generated torsion pair $(\Tcal,\Fcal)$ over a hereditary ring. Over a finite-dimensional tame hereditary algebra,
 the classes $\alpha(\Tcal)$ and $\beta(\Fcal)$ admit a further description. It turns out that they are precisely the wide coreflective subcategories of $\lMod\Lambda$  which are obtained as perpendicular categories to a collection of pure-injective modules. 

\begin{thm}\label{tamethm}
Let $(\Xcal, \Ycal)$ be a complete Ext-orthogonal pair over a finite-dimensional tame hereditary algebra  $ \Lambda $. The following statements are equivalent.
% for a full subcategory $ \mathcal{X} \subseteq \lMod{\Lambda} $.
\begin{enumerate}
\item There exists a wide subcategory $ \mathcal{W}$ of $\lmod{\Lambda} $ such that $ \mathcal{X} = \varinjlim \mathcal{W} $ or $ \mathcal{X} = \mathcal{W}^{\perp_{0,1}}$.
\item There exists a set of indecomposable pure-injective $\Lambda$-modules $ \mathcal{P} $ such that $ \mathcal{X} = {}^{\perp_{0,1}}\mathcal{P}$.
\item  $\Xcal$ or $\Ycal$ is bireflective.
\end{enumerate}
\end{thm}

The equivalence of (1) and (3) follows immediately from Theorem~\ref{wide and extorth}. The proof of the equivalence of (1) and (2) will be divided into  several steps to improve the overall clarity. We begin with the easy implication (1)$\Rightarrow$ (2):

\begin{lem}
Let $ \Lambda $ be a  tame hereditary algebra. Let $ \mathcal{W} $ be a wide subcategory of $ \lmod{\Lambda} $. Then there exist two families of indecomposable pure-injective $\Lambda$-modules $ \mathcal{P}, \mathcal{Q} $ such that $ \mathcal{W}^{\perp_{0,1}} = {}^{\perp_{0,1}}\mathcal{P} $ and $ \varinjlim \mathcal{W} = {}^{\perp_{0,1}}\mathcal{Q} $.
\end{lem}

\begin{proof}
% We obtain the result by using cosilting theory. Let $ \Tpair{T}{F} = (\mathbf{T}(\mathcal{W}), \mathcal{W}^{\perp_0}) $ be  the torsion pair in $ \lMod{\Lambda} $ generated by $\Wcal$; this is a cosilting torsion pair. In particular, we have that $ \alpha(\mathcal{T}) = \varinjlim \mathcal{W} $ by Proposition \ref{prop:alphaIsLimClosure} and that $ \beta(\mathcal{ F }) = \mathcal{W}^{\perp_{0,1}} $ by Lemma \ref{lem:betaCompHereditary}. 

%Consider now an approximation sequence $ 0 \to C_1 \to C_0 \to D\Lambda  $   as in (\ref{approxsequence}) with $C_0$ and $C_1$ in $\Prod C$ for the associated cosilting module $C$. Then  $ \varinjlim \mathcal{W} = {}^{\perp_{0,1}}C_0 $ by Remark \ref{perpendicular}.
% On the other hand, by the same remark, we have 
%$ \mathcal{W}^{\perp_{0,1}} = {}^{\perp_{0}}C_1 \cap \mathcal{F} $, however there exists an injective copresentation $ \sigma : I_0 \to I_1 $ of $ C_1 $ such that $ \mathcal{F} = \mathcal{C}_\sigma $ (see the proof of \cite[Proposition 3.5]{abundance}). A standard computation shows that any copresentation can be decomposed as $ \sigma = \mu \oplus \tau \oplus 0 \to J $, where $ \mu $ is a minimal injective copresentation of $ C_1 $ and $ \tau $ is an isomorphism. Since $ C_1 $ has injective dimension at most 1, we conclude that $ \mathcal{W}^{\perp_{0,1}} = {}^{\perp_{0}}C_1 \cap {}^{\perp_1} C_1 \cap {}^{\perp_0} J = {}^{\perp_{0,1}}(C_1 \oplus J) $.   

We obtain the result by using cosilting theory and the AR-formula.

For a category of the form $ \mathcal{W}^{\perp_{0,1}} $, recall that we have $ \mathcal{W} = \operatorname{filt}(\mathcal{B}) $ for a semibrick $ \mathcal{B} $ and $  \mathcal{W}^{\perp_{0,1}} =  \mathcal{B}^{\perp_{0,1}} $. Then if $ \mathcal{B} $ contains any indecomposable projective module $ P $, we have a corresponding indecomposable injective  $ I $ such that $ P^{\perp_{0,1}} = P^{\perp_{0}} = {}^{\perp_{0}} I = {}^{\perp_{0,1}} I $. For all the non-projective bricks $ B $ we have that $ B^{\perp_{0,1}} = {}^{\perp_{0,1} }\tau B $. In conclusion, we can find a set of indecomposable finite-dimensional modules $ \mathcal{B}' $ such that $ \mathcal{W}^{\perp_{0,1}} = {}^{\perp_{0,1}} \mathcal{B}' $.   

 For the case $ \varinjlim \mathcal{W} $, consider $ \Tpair{T}{F} = (\mathbf{T}(\mathcal{W}), \mathcal{W}^{\perp_0}) $ the torsion pair in $ \lMod{\Lambda} $ generated by $\Wcal$; this is a cosilting torsion pair. In particular, we have that $ \alpha(\mathcal{T}) = \varinjlim \mathcal{W} $ by Proposition \ref{prop:alphaIsLimClosure}.

Consider now an approximation sequence $ 0 \to C_1 \to C_0 \to D\Lambda  $   as in (\ref{approxsequence}) with $C_0$ and $C_1$ in $\Prod C$ for the associated cosilting module $C$. Then  $ \varinjlim \mathcal{W} = {}^{\perp_{0,1}}C_0 $ by Remark \ref{perpendicular}.

Now recall that $C$ is pure-injective, hence so are $C_0$ and $C_1$. Moreover, it is well known that over a tame hereditary algebra every pure-injective module $E$ is the pure-injective hull of  $\coprod E_i$, where the  $E_i$ are, up to isomorphism, precisely the indecomposable direct summands of $E$, see e.g.~\cite[Proposition 2.1]{BK}. Therefore the perpendicular category of $E$ is determined by a family of indecomposable pure-injective $\Lambda$-modules. This completes the proof.
\qedhere

\end{proof}

For the converse implication  we will need some classification results over tame hereditary algebras. Let us first recall the shape of the Auslander-Reiten quiver of
$\Lambda$. It consists, as in  Example~\ref{betalpha}, 
of a preprojective and a preinjective component, denoted by  $\mathbf{p}$ and $\mathbf{q}$, respectively, and    a  family of orthogonal tubes
$\mathbf{t}=\bigcup_{\lambda\in\mathbb{X}}\mathbf{t}_\lambda$ containing the  regular modules. 

Almost all tubes have rank 1. 
Given an exceptional tube $\mathbf{t}_{\lambda}$ of rank $r> 1$ and a module $X=U[m]\in\mathbf{t}_{\lambda}$ of regular length $m< r$, we may consider the full subquiver ${W}_{X}$ of $\mathbf{t}_{\lambda}$ which is isomorphic to the Auslander-Reiten-quiver  of the linearly oriented quiver of type $\mathbb{A}_{m} $ with $X$ corresponding to the projective-injective vertex.  The set ${W}_{X}$   is called a {\it wing} of $\mathbf{t}_{\lambda}$ \emph{of size $m$ with  vertex} $X$.

Next, we recall that there is a complete classification of  the indecomposable pure-injective $\Lambda$-modules: they are the indecomposable finite dimensional modules, the adic modules $S[-\infty]$ and  Pr\"ufer modules $S[\infty]$ corresponding to  simple regular modules $S$, and the generic module $G$.

The infinite-dimensional  cotilting modules were classified in    \cite{BK}. They are parametrized by pairs $(Y,P)$ where $Y$ is a branch module and $P$ is a subset of $\mathbb X$. More specifically, the following modules {form} a complete irredundant list of all large cotilting $\Lambda$-modules, up to equivalence:
$$C_{(Y,P)}=Y \oplus \coprod_{\mu\in P} \{\text{all}~S[\infty]~ \text{in}~ {^{\perp_{1}}Y}~\text{from}~\mathbf{t}_{\mu}\}  \oplus G \oplus
  {\prod_{\mu\notin P} \{\text{all}~S[-\infty]~ \text{in}~ {Y^{\perp_{1}}}~\text{from}~\mathbf{t}_{\mu}\}}$$
 
 Rather than giving the precise definition of a branch module, let us focus on the two 
particular cases which will be relevant for our discussion. We will choose the branch module $Y$ to consist of modules on a fixed ray or on a fixed coray in an exceptional tube $ \mathbf{t}_\lambda $. Given a quasi-simple module $S$, let us denote by $S[i]$ the  module of regular length $i$ on the ray starting in $S$, and by $ S[-i] $  the  module of regular length $i$ on the coray ending at $S$.
%We can choose  the branch module $Y=S[1] \oplus \dots \oplus S[r - 1]$ to consist of the first $r-1$ modules on the  ray starting in $S$,    or on a fixed coray $\{S[-n]\mid n\in\mathbb N\}$ in an exceptional tube $ \mathbf{t}_\lambda $, and 
Setting $P$ to be the singleton containing $\lambda$, or its complement, we obtain the following cotilting modules.

\begin{prop}
Let $ \Lambda $ be a tame hereditary algebra.  Given an index $ \lambda \in \mathbb{X} $ such that the tube $ \mathbf{t}_\lambda $ has rank $ r $, and a quasi-simple $ S \in \mathbf{t}_\lambda $, the following modules are cotilting:
\begin{enumerate}
\item $ C_{S}^{+} := S[1] \oplus \dots \oplus S[r - 1] \oplus S[\infty] \oplus G \oplus \prod_{ \mu \neq \lambda} M_\mu $, where  each $ M_\mu $ is the direct sum of all the adic modules from the tube $ \mathbf{t}_\mu $.

\item $ C_{S}^{-} := S[-\infty] \oplus S[-r + 1] \oplus \dots \oplus S[-1] \oplus G \oplus \coprod_{ \mu \neq \lambda} N_\mu $, where each $ N_\mu $ is the direct sum of all the Pr\"{u}fer modules from the tube $ \mathbf{t}_\mu $.
\end{enumerate}
\end{prop}

\begin{rem}\label{details}
{\rm
(1) By \cite[Theorem 3.16]{AC}, the modules $ C^+_S $ and $ C^-_S $ are \emph{minimal} cotilting modules, and  the corresponding torsion pairs are widely generated by Proposition~\ref{prop:MinCosOverHeredIsWide}.

\smallskip

(2)
In Figure \ref{figure:TpForCplus} we sketch the torsion pair $(\Tcal,\Fcal)$ cogenerated by the cotilting module $ C^+_S $.

\begin{figure}
\label{figure:TpForCplus}
\centering
\begin{tikzpicture}[scale=0.70]
\fill [black!50, rounded corners] (5.6, -1.5) -- (5.6, 6) -- (15,5) -- (17.2,-1.5) -- cycle;
\fill [blue!30, rounded corners] ( -3.2, -1.5  ) -- (-3.2, 2) -- ( 2.4, 3  ) -- (2.4, -1.5 ) -- cycle;
\fill [black!50, rounded corners] ( 3.1, -1.2 ) -- (4, -0.5) -- (4.9, -1.2) -- (4.3, -1.4) -- (4, -1.5) -- (3.7, -1.4) -- cycle;

\draw (-3,0.5) to [in=180, out=90 ] (1.5,1.75);
\draw (-3,0.5) to [in=180, out=-90] (1.5,-0.75);
\draw ( 2, 0.5 ) node{$ \dots$};

\draw (3, -1) node[anchor=north]{$ S $};

\draw (3, -1) -- (3, 5);
\draw (5, -1 ) -- (5, 5);
\draw (4, -1) ellipse(28pt and 8pt);

\draw[ultra thick, blue!30] (3, -1) to [in=-90, out=30 ] (5, 2);
\draw[thick, dotted, blue!30] (5,2) to [out=90, in=-30] (3,5);

\draw (6,-1) -- (6,5);
\draw (6.5,-1) -- (6.5, 5);
\draw (6.25, -1) ellipse(7pt and 2pt);
\draw (7,-1) -- (7,5);
\draw (7.5,-1) -- (7.5, 5);
\draw (7.25, -1) ellipse(7pt and 2pt);
\draw (8, 0.5) node{$\dots$};
\draw (8.5,-1) -- (8.5,5);
\draw (9,-1) -- (9, 5);
\draw (8.75, -1) ellipse(7pt and 2pt);
\draw (10, 0.5) node{$\dots$};

\draw (11,1.75) to [in=90, out=0 ] (15.5,0.5);
\draw (11,-0.75) to [in=-90, out=0] (15.5,0.5);

\draw (-1, 0.5) node{$ \mathbf{p} $};
\draw (13.5, 0.5) node{$ \mathbf{q} $};

\end{tikzpicture}
\caption{The torsion pair for the cotilting module $ C^+_S $. Torsionfree modules in blue, torsion modules in gray}
\end{figure}
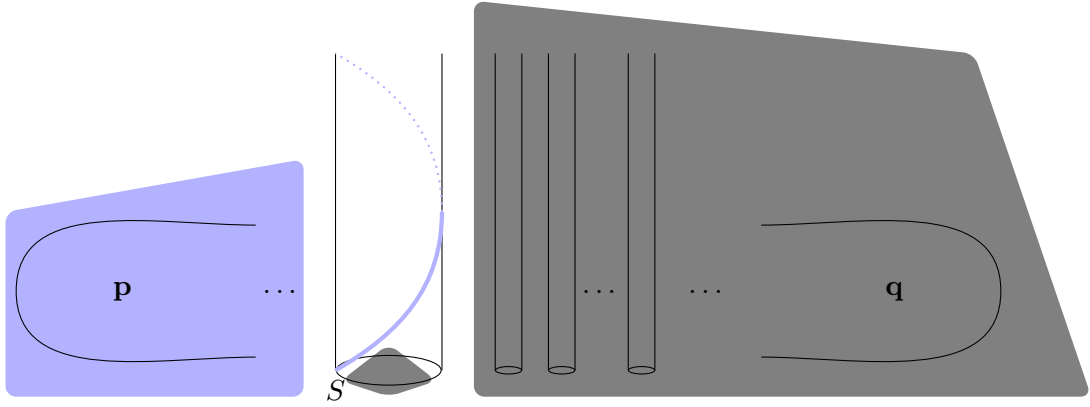

 Every preprojective module is torsionfree, while every preinjective module and every regular module in a tube different from  $ \mathbf{t}_\lambda $ is torsion. 
In the tube $ \mathbf{t}_\lambda $ we have that all  modules on the ray starting at $ S $ are torsionfree. The wide subcategory generated by the semibrick \mbox{$\Scal^+=\{ \tau^- S, \dots, \tau^{-r + 1} S = \tau S\}$} is contained in $\Tcal$. The other modules in the tube are neither torsion, nor torsionfree.  

Now it is easy to check that $\Wcal=\alpha(\Tcal)\cap\lmod{\Lambda}$  is given by the non-preinjective torsion modules. Therefore the torsion, almost torsionfree modules are precisely the quasi-simples in the tubes different from  $ \mathbf{t}_\lambda $  together with the modules in the semibrick $\Scal^+$, cf.~Proposition~\ref{prop:simpInalphabeta}.

Next, we claim that the brick $S[r]$ is torsionfree, almost torsion. Indeed, it is torsionfree, and all its proper quotients are  torsion: they are either preinjective, or they are modules on the coray ending at $ \tau S $ of regular length at most $ r - 1 $. 
%$(\Tcal,\Fcal)$ has a single torsionfree, almost torsion module, namely the brick $ S[r] $. It is torsionfree, and all its proper quotients are  torsion: they are either preinjective, or they are modules on the coray ending at $ \tau S $ of regular length at most $ r - 1 $. 
Consider a short exact sequence with torsionfree middle term
$
0 \to S[r] \to F \to N \to 0 
$.
By Remark~\ref{prop:finInf}, we may assume $ F $ is finite-dimensional.
In fact, since $ S[r] $ doesn't have any non-zero map to a preprojective module or to a regular module in a different tube,  we may assume that $ F $ is a direct sum of modules of the form $ S[i + r] $ for some $ i \in \mathbb{N} $.

As  $\add \mathbf{t}_\lambda $ is closed under cokernels, the module $ N $ decomposes as a direct sum of  modules in $ \mathbf{t}_\lambda $. Suppose $N$ is not torsionfree. By taking a pullback, we can reduce to the case where $ N $ is torsion and indecomposable,
%, so that $ N \simeq \tau^m S[k] $ for some $ m, k \in \mathbb{N} $. 
 that is,  it lies in the wing determined by the semibrick $\Scal^+$. However, none of these modules is obtained as a quotient of a torsionfree module by $S[r]$, as $ \Ext^1(N, S[r]) = 0 $ for all these $ N $. Hence $N$ must be torsionfree, and our claim is proven.
 
In fact, 
%\reml{oppure altro argomento per l'unicit\`a del tf/t?}
$S[r]$ is the unique torsionfree, almost torsion module: any such module must be a brick and must belong to $\beta(\Fcal)$ by Proposition~\ref{prop:simpInalphabeta}; however,  the  torsionfree bricks different from $S[r]$  are either of the form  $ S[i] $ with $ i < r  $, 
%and thus have a sequence $ 0 \to S[i] \to S[r] \to \tau^{-i}[r - i] \to 0 $ where the  $\tau^{-i}[r - i]$ is not torsionfree, 
or they are preprojective, and in both cases they admit a non-zero map to $ S[r] $ whose cokernel can't be torsionfree, cf.~\cite[XII, Lemma 3.6]{SS}.  

% e.g. 
%by using that $\beta(\Fcal)=\Wcal^{\perp_{0,1}}$ is a module category with a unique simple module. 

\smallskip

(3)
For the torsion pair cogenerated by $ C^-_S $ we have again that every preprojective module is torsionfree and that every preinjective module is torsion; moreover every regular module in a tube different from  $ \mathbf{t}_\lambda $ is torsionfree, while in  $ \mathbf{t}_\lambda $ the modules on the coray ending at $ \tau^-S $ are torsion and the modules in the wide subcategory generated by the semibrick $\Scal^-=\{ S, \tau S, \dots \tau^{r - 2} S \}$ are torsionfree. The other modules in the tube are neither torsion, nor torsionfree.

With similar arguments to (2) we see that the torsionfree, almost torsion modules are precisely the quasi-simples in the tubes different from  $ \mathbf{t}_\lambda $  together with the modules in the semibrick $\Scal^-$, and that 
there is a unique torsion, almost torsionfree module, namely the module $ \tau^- S [-r] $.
}
\end{rem}

We start the proof of the implication (2)$\Rightarrow$(1) by computing the perpendicular categories of the indecomposable pure-injective modules.

\begin{lem}
Let $ M $ be a finite-dimensional indecomposable module. Then $ {}^{\perp_{0,1}}M = \mathcal{W}^{\perp_{0,1}} $ for some $ \mathcal{W}$ in  $\mathbf{wide}(\Lambda)$.
\end{lem}

\begin{proof}
We distinguish two cases. If $ M $ is injective, then $ {}^{\perp_{0,1}}M = {}^{\perp_0}M $ is a Serre subcategory and there exists an indecomposable projective $ P $ such that $ P^{\perp_0} = {}^{\perp_0}M $ so that we can take $ \mathcal{W} = \operatorname{filt}(P) $. 
If $ M $ is not injective, we use the AR-formulae to obtain
$
{}^{\perp_{0,1}}M = \tau^{-}M^{\perp_{0,1}}
$
Thus we can take $ \mathcal{W} $ to be the smallest wide subcategory containing $ \tau^{-}M $. 
\end{proof}

\begin{lem}\label{G}
For the generic module $ G $ we have
$
{}^{\perp_{0,1}}G = \varinjlim \Wcal
$
where $\Wcal=\add \mathbf{t} $ is the wide subcategory of $ \lmod{\Lambda} $ spanned by the regular modules.
\end{lem}
\begin{proof}
We consider the cotilting torsion pairs $(\Gen\mathbf{t},\mathbf{t}^{\perp_0})$ 
 and $(\Gen\mathbf{q},\mathbf{q}^{\perp_0})$ generated by the regular and by the preinjective modules, respectively. The first one is given by the cotilting module   $C_{(0,\emptyset)}=G \oplus
  {\prod_{\mu\in \mathbb{X}} \{\text{all}~S[-\infty]~\text{from}~\mathbf{t}_{\mu}\}}$, and we infer that 
$(\Gen\mathbf{t},\mathbf{t}^{\perp_0})=({}^{\perp_{0}}G,\Cogen G)$. The second one is given by $C_{(0,\mathbb{X})}=G \oplus
  {\prod_{\mu\in \mathbb{X}} \{\text{all}~S[\infty]~\text{from}~\mathbf{t}_{\mu}\}}$. If  $ 0 \to C_1 \to C_0 \to D\Lambda \to 0 $   is  a minimal approximation sequence as in (\ref{approxsequence}), then we know from  \cite[Theorem 7.1]{RR} that  $C_1$ is a direct sum of copies of $G$, thus the cotilting class  $\Ccal=\mathbf{q}^{\perp_0}={}^{\perp_{1}}C_1$ equals ${}^{\perp_{1}}G$.  

We conclude that ${}^{\perp_{0,1}}G= \Gen\mathbf{t}\cap \Ccal$, which has the stated shape by \cite[\S 3.4]{RR}.
\end{proof}

For the Pr\"{u}fer and adic modules we will need some results from 
\cite{simpleCotilting} which allow us to locate  $S[\infty]$ and $S[-\infty]$ inside the minimal approximation sequences  given by the cotilting modules  $ C_S^+$ and $C_S^-$.

\begin{de}
Let $C$ be a cotilting module with associated  torsion pair $  \Tpair{T}{F} $. 
A module $ E \in \Prod C $ is called 
\emph{critical} if there exists a 
 short exact sequence
$0 \to F \to E \xrightarrow{b} M \to 0 $
where  $ F $ is a torsionfree, almost torsion module with respect to $\Tpair{T}{F}$, and $ E $ is the $ \Prod(C)-$envelope of $ F $. 
Moreover, a module $ E \in \Prod C $ is called \emph{special} if there exists a
 short exact sequence:
$
0 \to E \xrightarrow{b} N \to T \to 0 
$
where  $ T $ is a torsion, almost torsionfree module with respect to $ \Tpair{T}{F} $, and $N$ is the $ \mathcal{F}-$cover of $ T $. 
\end{de}

\begin{rem}
{\rm
This definition differs from, but is equivalent to the original definition in \cite{simpleCotilting}. There critical and special modules are defined in terms of the existence of certain (strong) left almost split maps in $ \mathcal{F} $. As shown in \cite[Lemma 4.3]{simpleCotilting}, such left almost split maps are either injective or surjective. The critical modules  are the modules in $\Prod C$ which are source of a left almost split epimorphism in $\Fcal$, and the special modules are the modules in $\Prod C$ which are source of a left almost split monomorphism in $\Fcal$, see \cite[Corollaries 5.18 and 5.22]{simpleCotilting}. By \cite[Theorem 4.2]{simpleCotilting} these are precisely the modules defined above. In fact, the maps $b$ in the definition are the required left almost split morphisms. 
}
\end{rem}

\begin{prop}[{\cite[Corollary 5.23 and Lemma 6.10]{simpleCotilting}}]
Let $ C $ be a cotilting module with associated  torsion pair $  \Tpair{T}{F} $ and minimal approximation sequence
$
0 \to C_1 \to C_0 \to D\Lambda \to 0
$
as in (\ref{approxsequence}).
Then every special  module is a direct summand of $ C_1 $ and every critical  module is a direct summand of $ C_0 $.
Moreover, if $C$ is a minimal cotilting module, then an indecomposable module lies in $\Prod (C)$ if and only if it lies in $\Prod (C_0)$ or $\Prod (C_1)$, and not in both.
\end{prop}

Now we turn to the minimal cotilting modules  $ C_S^+$ and $C_S^-$. We will see that the Pr\"ufer module $S[\infty]$ is the only  critical  summand of $ C_S^+$, and the adic module $S[-\infty]$ is the only special  summand of $ C_S^-$. This facts will be exploited to compute the perpendicular categories of  $S[\infty]$ and $S[-\infty]$.

\begin{lem}
Let $ S $ be a quasi-simple module in a tube of rank $ r $, and let $ S[\infty] $ be the corresponding Pr\"{u}fer module. Then there exists a wide subcategory $ \mathcal{W}$ in $\mathbf{wide}(\Lambda)$ which consists of regular modules and satisfies
$
{}^{\perp_{0,1}} S[\infty] = \varinjlim \mathcal{W} 
$
\end{lem}

\begin{proof}
Denote by $(\Tcal,\Fcal)$ the torsion pair associated to $C_S^+ $.
We will show that $ {}^{\perp_{0,1}} S[\infty] = \alpha(\Tcal) $. To this aim, it is enough to show that in the minimal approximation sequence
$
0 \to C_1 \to C_0 \to D\Lambda \to 0 
$
for $C_S^+ $ we have $ \Prod(C_0) = \Prod(S[\infty]) $. 
This amounts to show that $ S[\infty] $ is the unique critical module and that all the other summands of $ C_S^+ $ are either special or occur as summands in products of copies of $ S[\infty] $. 

The latter is the case for the generic module $ G $, which is known to lie in $ \Prod(S[\infty] ) $ and thus satisfies $ {}^{\perp_{0,1}}S[\infty] = {}^{\perp_{0,1}} ( S[\infty] \oplus G ) $. 
 
%Now, this particular torsion pair has a single torsionfree, almost torsion module, namely the brick $ S[r] $. This is torsionfree, as it is a submodule of $ S[\infty] $, moreover, all its proper quotients are in the torsion class: they are either preinjective modules or modules in the coray starting at $ \tau S $ of regular length at most $ r - 1 $. Moreover, assume we have a short exact sequence with torsionfree middle term:\[0 \to S[r] \to F \to N \to 0 \]Using the compatibility results, we may assume $ F $ is finite-dimensional.Notice how $ S[r] $ doesn't have any non-zero map to a preprojective module or to a regular module in a different tube, thus we may assume that $ F $ is a direct sum of modules of the form $ S[i + r] $ for some $ i \in \mathbb{N} $.As the tube in itself is closed under cokernels, the module $ N $ decomposes as a sum of regular modules in the tube. By taking pullbacks, we can reduce to the case were $ N $ is indecomposable, so that $ N \simeq \tau^m S[k] $ for some $ m, k \in \mathbb{N} $. In fact, we may assume that $ N $ is a torsion module, i.e. that it lies in the wing containing the quasi-simples $ \tau^-S, \dots, \tau^{- r + 1} S $. However, none of these modules is obtained as a quotient of a torsionfree module by $ S[r] $, in fact $ \Ext^1(N, S[r]) = 0 $ for all these $ N $.

Next, we show that $ S[\infty] $ is critical. We have seen in Remark~\ref{details} that $S[r]$ is the (unique) torsionfree, almost torsion module, so we have to show that $ S[\infty] $  is its  $ \Prod(C^+_S)-$envelope. This is immediate, in fact $ S[\infty] $ is the only indecomposable module in $ \Prod(C^+_S) $ with a non-zero morphism from $ S[r] $, and
the envelope is given by the short exact sequence
\[
0 \to S[r] \to S[\infty] \xrightarrow{b} S[\infty] \to 0.
\]

It remains to show  that all the other summands of $ C^+_S $ are special. For the finite-dimensional ones, this is witnessed by the short exact sequences $$ 0\to S[i] \to S[i + 1] \xrightarrow{g_i} \tau^{-i} S \to 0 , \: 1\le i<r.$$ Notice that $g_i$ is an $\Fcal$-cover of the torsion, almost torsionfree module  $\tau^{-i} S$, 
because  $S[i]$ lies in $\Prod(C^+_S)$ and is therefore Ext-injective in $\Fcal$.
For the adic summands, we use the fact that every quasi-simple module in a tube different from $\mathbf{t}_\lambda $ is  torsion, almost torsionfree,
%In fact, any proper submodule of $ S' $ is a preprojective, thus torsionfree module. Moreover, for every short exact sequence in $ \lmod{\Lambda} $\[0 \to K \to T \to S' \to 0 \] we have that $ T = Q \oplus R \oplus T' $ where $ Q $ is a direct sum of preinjective modules, $ R $ is a direct sum of regular modules from tubes not containing $ S' $ and $ T' $ is a sum of regular modules from the tube containing $ S' $. It follows that the sequence can be rewritten as:\[\begin{tikzcd}0 \ar[r] & (Q \oplus R) \oplus K' \ar[r, "f"] & (Q \oplus R) \oplus T' \ar[r] & S' \ar[r] & 0\end{tikzcd}\]Where $ f = Id_{Q \oplus R} \oplus k' $, thus we only need to show that the kernel $ K' $ of the map $ T' \to S' $ is torsion, but this is clear as both the modules $ T' $ and $ S' $ are contained in the same tube, so the kernel $ K' $ is in the tube as well. 
and that every adic summand appears as the kernel of the $ \mathcal{F}-$cover of one of these modules. This is witnessed by the sequence
\[
0 \to S'[-\infty] \to \tau^- S'[-\infty] \xrightarrow{g} \tau^- S' \to 0
\]
where  $g$ is again an $\Fcal$-cover
because $  S'[-\infty] $ is Ext-injective in $ \mathcal{F}$.

We have shown that $ \Prod(C_0) = \Prod(S[\infty]) $, and we can conclude by Proposition \ref{prop:alphaIsLimClosure} and Remark~\ref{perpendicular} that $$ {}^{\perp_{0,1}}{S[\infty]} = {}^{\perp_{0,1}}C_0=\alpha(\Tcal)=\varinjlim\Wcal $$ where $ \mathcal{W} = \alpha(\Tcal)\cap\lmod{\Lambda}={}^{\perp_{0,1}}S[\infty] \cap \lmod{\Lambda} $ consists of regular modules, as already observed in  Remark~\ref{details}.
\end{proof}

\begin{lem}
Let $ S $ be a quasi-simple module in a tube of rank $ r $, and let $ S[-\infty] $ be the corresponding adic module. Then there exists a wide subcategory $ \mathcal{W}$ in $\mathbf{wide}(\Lambda)$ which consists of regular modules and satisfies 
$
{}^{\perp_{0,1}} S[-\infty] =  \mathcal{W}^{\perp_{0,1}} 
$
\end{lem}

\begin{proof}
We proceed dually to the Pr\"{u}fer case; here we want to show that $ S[-\infty] $ is the unique special summand of the cotilting module $ C_S^- $. 

We consider  the torsion pair $(\Tcal,\Fcal)$  cogenerated by $C_S^-$ and a minimal approximation sequence
$
0 \to C_1 \to C_0 \to D\Lambda \to 0 
$.
First, we notice that all the finite-dimensional summands  and all the Pr\"{u}fer summands of $ C_S^- $ are contained in $ \Prod(C_0) $. In fact $ C_0 $ cogenerates the torsionfree class $\Fcal$, but none of these modules is cogenerated by the other summands of $ C_S^- $.  
Then also the generic module $ G $, being  a direct summand in a direct product of copies of any Pr\"{u}fer module, lies in $ \Prod(C_0) $.

In order to prove that $ S[-\infty] $ is special, we consider the (unique) torsion, almost torsionfree module $ \tau^- S [-r] $.
From the pullback diagram
\[
\begin{tikzcd}
0 \ar[r] & \tau^- S[-\infty] \ar[r] \ar[d, equals] & P \ar[r] \ar[d, hook] & S[-r + 1] \ar[d, hook] \ar[r] & 0 \\
0 \ar[r] & \tau^- S[-\infty] \ar[r] & \tau^- S[-\infty] \ar[d, two heads] \ar[r] & \tau^- S[-r] \ar[d, two heads] \ar[r] & 0 \\
& & \tau^- S \ar[r, equals] & \tau^- S
\end{tikzcd}
\]
we obtain that $ P \simeq S[-\infty] $. Thus we have a short exact sequence
\[
0 \to S[-\infty] \to \tau^- S [-\infty] \oplus S[-r + 1] \xrightarrow{g} \tau^- S [-r] \to 0 
\]
that represents the adic as kernel of the $ \mathcal{F}$-cover of a torsion, almost torsionfree module. Here $g$ is the map induced by the canonical projection $ \tau^- S [-\infty] \to \tau^- S[-r] $ and by the irreducible morphism $ S[-r + 1] \to \tau^- S[-r] $. It is indeed   an $\Fcal$-precover because  $S[-\infty]$ is Ext-injective in $\Fcal$, and it minimal as its components are non-trivial.

In conclusion, recalling that $(\Tcal,\Fcal)$ is widely generated, and combining Remark~\ref{perpendicular} with  Lemma \ref{lem:betaCompHereditary}, we obtain
\[
{}^{\perp_{0,1}}S[-\infty] = {}^{\perp_{0,1}}C_1 = \beta(\Fcal) =\mathcal{W}^{\perp_{0,1}} 
\]
where $ \mathcal{W} =  \alpha(\Tcal)\cap\lmod{\Lambda}=\operatorname{filt}(\tau^-S[-r]) ) $ consists of regular modules.
\end{proof}

For the general case, when $\Pcal$ is  a collection of indecomposable pure-injective modules, we will make use of the following decomposition result, consequence of  work by  Ringel.

\begin{prop}[{\cite[Theorem in Section G, Theorem 4.4, Proposition 4.8  ]{Ringel}}]\label{R}
Every module $ M \in \varinjlim \add \mathbf{t} $ has a decomposition $ M = \coprod_{ \lambda \in \mathbb{X} } M_\lambda $ where $ M_\lambda $ is the largest submodule of $ M $ belonging to $ \varinjlim \add \mathbf{t}_\lambda $. Moreover, for each $ \lambda \in \mathbb{X} $ there exists a pure exact sequence
\[
0 \to A_\lambda \to M_\lambda \to Z_\lambda \to 0  
\]
where $ A_\lambda $ is a direct sum of finite-dimensional modules in $ \mathbf{t}_\lambda $, while $ Z_\lambda $ is a direct sum of Pr\"{u}fer modules from $ \mathbf{t}_\lambda $ .
\end{prop}

We will also need the following structure result for the wide subcategories lying in the additive closure of a tube.

\begin{lem}
\label{lem:wideStructureInTube}
Let $ \mathcal{A} = \add\mathbf{t}_\lambda $ be the additive closure of  a tube of rank $ r $, which we view as an abelian length category with simple objects  denoted by $ S_1, \, S_2=\tau^- S_1, ... ,  S_r = \tau^{-r + 1} S $. The non-trivial wide subcategories of $ \mathcal{A} $ have the shape $ \add(\mathcal{F} \cup \mathcal{T}) $ where $ \mathcal{F} $ is a subset of a wing of size at most  $r-1$, and $ \mathcal{T}$ is either zero or  $ \mathcal{T} = \{ S_i[nr] \mid n \in \mathbb{N}  \} $.
\end{lem}

\begin{proof}
Recall that  wide subcategories in an abelian length category are in bijection with semibricks. 
Thus, we consider the bricks in the category $ \mathcal{A} $. These are precisely the modules $ S_i[k] $ for $ 1 \le i,k \le r $.

Notice that  every brick $ S_i[r] $ has non-zero morphisms to any   other brick $ S_j[r] $. In particular, a given semibrick can only contain a single brick of regular length $ r $.
Moreover, using the structure of the tube, we see that the  bricks in $\Acal$ which are Hom-orthogonal to a given $ S_i[r] $ are contained in the wing with vertex $ \tau^- S_i[r - 2] $ (for $ r \ge 3 $ ). 
However, all the modules in this wing are Ext-orthogonal to $ S_i[r] $. Therefore a wide subcategory of $\Acal$ containing $ S_i[r] $ is the additive closure of some collection of indecomposables from the wing with vertex $ \tau^- S_i[r - 2] $ and of the self-extensions of $ S_i[r] $. The latter are precisely the modules $ S_i[nr] $. 

It remains to consider the wide subcategories whose semibricks do not contain bricks of regular length $r$. In this case, the semibrick, and thus the whole wide subcategory, must be  contained in a wing of size at most $r-1$.
\end{proof}

\begin{lem}
Let $ \mathcal{P} $ be a collection of indecomposable pure-injectives consisting of Pr\"{u}fer modules and possibly also of the generic module. Then
$
{}^{\perp_{0,1}} \mathcal{P} = \varinjlim \mathcal{W} 
$
for some wide subcategory in $\mathbf{wide}(\Lambda)$ consisting of regular modules.
\end{lem}

\begin{proof}
If the set $ \mathcal{P} $ consists only of the generic module, then we are done by Lemma~\ref{G}. If it contains at least a Pr\"{u}fer module, then the generic doesn't contribute to the computation of the perpendicular category, thus we may assume it doesn't belong to $ \mathcal{P} $. 

We want to show that for two wide subcategories consisting of regular modules $ \mathcal{U} $ and $ \mathcal{V} $ we have $ \varinjlim \mathcal{U} \cap \varinjlim \mathcal{V} = \varinjlim (\mathcal{U} \cap \mathcal{V} ) $. 
The inclusion $\supseteq$ is immediate. For the reverse inclusion, pick $ X \in \varinjlim \mathcal{U} \cap \varinjlim \mathcal{V} $. Without loss of generality, we may assume that $ X \in \varinjlim \add \mathbf{t}_\lambda $ for some index $ \lambda $. Then by Proposition~\ref{R} this module fits into a pure short exact sequence 
\[
0 \to A_\lambda \to X \to Z_\lambda \to 0
\]
and since all the three subcategories involved are closed under kernels, pure-epimorphic images and extensions, $ X $ lies in one category if and only if so do $ A_\lambda $ and $ Z_\lambda $. 

But now, $ A_\lambda $ is a coproduct of finite-dimensional modules, thus $ A_\lambda \in \varinjlim \mathcal{U} \cap \varinjlim \mathcal{V} $ if and only if $ A_\lambda \in \Add(\mathcal{U} \cap \mathcal{V}) \subseteq \varinjlim (\mathcal{U} \cap \mathcal{V}) $.
As for $ Z_\lambda $,  it is a coproduct of Pr\"{u}fer modules, so once again, we may restrict to the case of a single Pr\"{u}fer module.
 But $ S[\infty] $ can only be written as a direct limit of regular modules if we use infinitely many modules from the ray starting at $ S $. However, using Lemma \ref{lem:wideStructureInTube}, we see that the intersection of a wide subcategory of $ \lmod{\Lambda} $ with a tube is either contained in some wing of size strictly smaller than the rank of the tube, or it contains all modules  of regular length $nr$, $n\in\mathbb{N}$, on a certain ray in that tube. This consideration yields that $ S[\infty] \in \varinjlim \mathcal{U} \cap \varinjlim \mathcal{V} $ if and only if the modules $ S[nr], n \in \mathbb{N}$, are in $ \mathcal{U} \cap \mathcal{V} $, so that $ S[\infty] \in \varinjlim (\mathcal{U} \cap \mathcal{V}) $.
\end{proof}

For collections of adic and finite-dimensional modules we  use a result due to Schofield.

\begin{lem}
Let $ \mathcal{P} $ be a collection of finite-dimensional or adic modules. Then there exists a wide subcategory $ \mathcal{W} $ in $\mathbf{wide}(\Lambda)$ such that
$
{}^{\perp_{0,1}}\mathcal{P} = \mathcal{W}^{\perp_{0,1}}.$
\end{lem}

\begin{proof}
By
\cite[Theorem 2.3]{uniLocHered}, \cite[Theorem 6.1]{extPairs}, the assignment $\Wcal\mapsto \mathcal{W}^{\perp_{0,1}} $ defines a bijection between  $\mathbf{wide}(\Lambda)$  and the extension-closed bireflective subcategories of $ \lMod{\Lambda} $.
Notice that $ {}^{\perp_{0,1}} \mathcal{P} $ is of the form $ \bigcap_{P \in \mathcal{P}} \mathcal{W}_P\,^{\perp_{0,1}} $ with each $ \mathcal{W}_P $  in $\mathbf{wide}(\Lambda)$. This is an extension-closed bireflective subcategory of $ \lMod{\Lambda} $, as so are all $\mathcal{W}_P\,^{\perp_{0,1}}$. Hence it has the required shape.
\end{proof}

Now we are ready for the final step.

\begin{lem}
Let $ \mathcal{P} $ be a collection of pure-injective modules. Then there exists a wide subcategory $ \mathcal{W} $ of $\lmod{\Lambda}$ such that $ {}^{\perp_{0,1}}\mathcal{P} = \varinjlim \mathcal{W} $ or $ {}^{\perp_{0,1}} \mathcal{P} = \mathcal{W}^{\perp_{0,1}} $.
\end{lem}

\begin{proof}
By the discussion above, we only have to treat the  case when $ \mathcal{P} = \mathcal{A} \cup \mathcal{D} $ where $ \mathcal{A} $ consists of finite-dimensional or adic modules, and $ \mathcal{D}$ consists of Pr\"{u}fer modules (and possibly the generic). 
By our previous computations, we have that $ {}^{\perp_{0,1}}\mathcal{A} = \mathcal{U}^{\perp_{0,1}} $ and $ {}^{\perp_{0,1}} \mathcal{D} = \varinjlim \mathcal{W} $ where $\Ucal$ and $\Wcal$ are in $\mathbf{wide}(\Lambda)$ and $ \mathcal{W} $ consists of regular modules. Then  $ {}^{\perp_{0,1}} \mathcal{P} = \mathcal{U}^{\perp_{0,1}} \cap \varinjlim \mathcal{W} $, and 
we want to show that the latter coincides with the direct limit closure of the wide subcategory $ \mathcal{U}^{\perp_{0,1}} \cap \mathcal{W}$.
% in $\mathbf{wide}{\Lambda}$. 

The inclusion  $ \varinjlim (\mathcal{U}^{\perp_{0,1}} \cap \mathcal{W} ) \,\subseteq \,\mathcal{U}^{\perp_{0,1}} \cap \varinjlim \mathcal{W} $ is clear as $ \mathcal{U}^{\perp_{0,1}} $ is closed under direct limits.
For the other inclusion, assume $ X \in  \mathcal{U}^{\perp_{0,1}} \cap \varinjlim \mathcal{W} $. As this module is in $ \varinjlim \mathcal{W} \subseteq \varinjlim \add \mathbf{t}$, we can again apply Proposition~\ref{R}  to reduce to the cases when $ X \in \Add(\lmod{\Lambda}) $ or $ X $ is a coproduct of Pr\"{u}fer modules.
The first case is immediate; in the second case, we may further reduce to  $ X = S[\infty] $ by using that all the subcategories we are considering are closed under summands and coproducts. 
Then, as observed before, $ \mathcal{W} $ must contain the modules $ S[nr] $ for all $ n \in \mathbb{N} $.

Since $ S[\infty] \in \mathcal{U}^{\perp_0} $, the same holds true for  its submodules $ S[nr] $. 
We claim that the modules $ S[nr] $ also lie in $\mathcal{U}^{\perp_1}$.
To this end, we pick $ U \in \mathcal{U} $, which we may assume indecomposable.
Notice that $ U $ is neither preprojective, nor preinjective, because  it is Hom-orthogonal to $ S[\infty] $, see \cite[XII, Lemma 3.6]{SS}. Assume
  that $ \Ext^1(U, S[r]) = D\Hom(\tau^-S[r], U) \ne 0 $. Then the regular module $ U $  must   be of the form $U=\tau^i S[k]$ with $0 \le i \le r - 1$ and $k - i \ge 1$.
But for any such choice, we  have $ \Hom(U, S[\infty]) \ne 0 $, a contradiction. So we have proved our claim for $ S[r] $, and consequently, also for all $ S[nr] $, as  $ \mathcal{U}^{\perp_1} $ is closed under extensions.  
This shows that $ S[\infty] $ is a direct limit of modules in $\mathcal{U}^{\perp_{0,1}} \cap \mathcal{W}$, concluding the proof.
\end{proof}

\section{Wide coreflective subcategories over the Kronecker algebra}

In the previous section we have determined the wide coreflective subcategories which are perpendicular to collections of pure-injective modules.  The existence of further  wide coreflective subcategories seems to be an intriguing question. In fact, it is related to the problem of  classifying the localizing subcategories in the unbounded derived category $\Dcal(\lMod A)$. Let us start by recalling some terminology.

\begin{de}{\rm
Let $ \mathcal{ T } $ be a triangulated category with suspension functor $ \Sigma $.
 
(1)  Two subcategories $ \mathcal{U}, \mathcal{V} $ of $\Tcal$ closed under direct summands  form a \emph{torsion pair} $ ( \mathcal{U}, \mathcal{V}) $ if:
\begin{itemize}
\item[(i)] $ \Hom_{\mathcal{T}}(\mathcal{U}, \mathcal{V}) = 0 $
\item[(ii)] For all objects $ T \in \mathcal{ T } $ we can find a triangle
\[
\begin{tikzcd}
U_T \arrow[r] & T \arrow[r] & V_T \arrow[r] & \Sigma U_T
\end{tikzcd}
\]
with $ U_T \in \mathcal{U} $ and $ V_T \in \mathcal{V} $. 
\end{itemize}

If in addition $\mathcal{U}$ and $  \mathcal{V} $ are both closed under suspension, that is, they are triangulated subcategories of $\Tcal$, then  $ ( \mathcal{U}, \mathcal{V}) $ is called a \emph{stable t-structure} (or a \emph{semiorthogonal decomposition}).

(2) A full subcategory $\Ucal$ of $\Tcal$ is \emph{localizing} if it is a triangulated subcategory closed under taking coproducts,
% (and then it is also \emph{thick}, i.e.~closed under direct summands). 
and it is  \emph{strictly localizing} if it can be completed to a stable t-structure $ ( \mathcal{U}, \mathcal{V}) $ in $\Tcal$.}
\end{de}

Observe that taking zero cohomology yields a bijection between the localizing subcategories of 
$\Dcal(\lMod A)$ and the wide subcategories of $\lMod A$ that are closed under coproducts. The inverse map assigns to $\Xcal$ the category $\Dcal_\Xcal(\lMod A)$ consisting of the complexes in $\Dcal(\lMod A)$ with all cohomologies in $\Xcal$, cf.~\cite[Proposition 2.4]{extPairs}. These  correspondences restrict as follows.

\begin{prop} \cite{extPairs,Nakamura}
\label{widecoref}
Let $A$ be a hereditary ring. There are bijections between 
\begin{enumerate}
\item strictly localizing subcategories of $\Dcal(\lMod A)$;
\item complete Ext-orthogonal pairs in $\lMod A$;
\item wide coreflective subcategories of $\lMod A$.
\end{enumerate}
\end{prop}
\begin{proof}
It is shown in \cite[Proposition 2.7]{extPairs} that a full subcategory $\Xcal$ of $\lMod A$ can be completed to a complete Ext-orthogonal pair $(\Xcal,\Ycal)$ if and only if the category $\Dcal_\Xcal(\lMod A)$ is the kernel of a localization functor. By \cite[Proposition 1.6]{AJS} this amounts to the existence of a right adjoint for the inclusion functor $\textrm{inc}:\Dcal_\Xcal(\lMod A)\hookrightarrow\Dcal(\lMod A)$, or in other words, to the fact that $\Ucal=\Dcal_\Xcal(\lMod A)$ can be completed to a stable t-structure $(\Ucal,\Vcal)$. 
Furthermore, by a result of Nakamura \cite{Nakamura}, the functor $\textrm{inc}:\Dcal_\Xcal(\lMod A)\hookrightarrow\Dcal(\lMod A)$ admits a right adjoint if and only if so does $\textrm{inc}:\Xcal\hookrightarrow\lMod A$. This shows that the correspondences above restrict to a bijection between (1) and (3), and that there is a natural bijection between (2) and (3).
\end{proof}

%In \cite{KSte}, Krause and Stevenson discuss  the localizing subcategories in the derived category of quasicoherent sheaves on the projective line $\mathbb P^1$ over a field $k$. Combining  the classifications of smashing subcategories and tensor ideals, they obtain a class of strictly localizing subcategories which are parametrized by a copy of $\mathbb Z$ and the powerset of $\mathbb P^1$, and they show that these admit an intrinsic characterization as kernels of cohomological functors, or equivalently, as left perpendicular categories of  collections of indecomposable pure-injective  quasi-coherent sheaves, see \cite[Corollary 4.3.3]{KSte}. They then ask whether there are further strictly localizing subcategories.

The  localizing subcategories of the derived category of a commutative noetherian ring were completely classified in work of Hopkins and Neeman \cite{N}; they are parametrized by subsets of the prime spectrum. But already for  the simplest non-affine case,  the derived category $\Dcal({\rm Qcoh}\,\mathbb P^1_k)$ of  the category of quasicoherent sheaves on the projective line $\mathbb P^1_k$ over an algebraically closed field $k$, the situation appears to be rather intricate. 
In \cite{KSte}, Krause and Stevenson address the problem of classifying  the strictly localizing subcategories of $\Dcal({\rm Qcoh}\,\mathbb P^1_k)$, that is, the localizing subcategories $\Lcal$  appearing in semiorthogonal decompositions $(\Lcal,\Mcal)$ of $\Dcal({\rm Qcoh}\,\mathbb P^1_k)$. Combining  the classifications of smashing subcategories and tensor ideals, they obtain a class of strictly localizing subcategories of $\Dcal({\rm Qcoh}\,\mathbb P^1_k)$ which are parametrized by a copy of $\mathbb Z$ and the powerset of $\mathbb P^1$, and they ask  whether all strictly localizing subcategories arise in this way.

This problem can be phrased inside the derived category of the Kronecker algebra $\Lambda$ via the well-known derived equivalence between the Kronecker quiver and the projective line. Krause and Stevenson have given an intrinsic description of  their class of strictly localizing subcategories in terms of perpendicular categories of pure-injective sheaves. The question then becomes whether there are strictly localizing subcategories in $\Dcal(\lMod \Lambda)$ which are not of the form ${}^{\perp_{\mathbb{Z}}}\Pcal=\{X\in\Dcal(\lMod \Lambda)\mid \Hom_{\Dcal(\lMod \Lambda)}(X, Y[i])=0 \text{ for all } i\in\mathbb{Z} \text{ and all } Y\in\Pcal\}$ for a collection $\Pcal$ of indecomposable pure-injective $\Lambda$-modules. 
%We now carry this classification problem from the category $\Dcal({\rm Qcoh}\,\mathbb P^1_k)$ to the category  $\lMod\Lambda$ of modules over the path algebra $\Lambda$ of the Kronecker quiver {\small \begin{tikzcd}0 \ar[r, bend left] \ar[r, bend right] & 1\end{tikzcd}} and translate it into a classification problem for wide coreflective subcategories.  Our problem then becomes \medskip {\bf Question:} Is it true that every wide coreflective subcategory $\Xcal$ of $\lMod\Lambda$ is the (left) perpendicular category ${}^{\perp_{0,1}} \Pcal$ of a collection of indecomposable pure-injective modules $\Pcal$? \medskip
%We show (Theorem~\ref{perppi}) that $\Xcal$ has this shape if and only if it arises from a  wide subcategory $\Wcal$ of the category $\lmod\Lambda$ of finite dimensional $\Lambda$-modules by some standard constructions. More precisely,  $\Xcal$ is either the (right) perpendicular category $\Wcal^{\perp_{0,1}}$ or the direct limit closure $\varinjlim \Wcal$ of $\Wcal$.
In virtue of Proposition~\ref{widecoref}, this amounts to asking 
\begin{que}{\rm 
Are there wide coreflective subcategories of $\lMod\Lambda$  which are not 
of the form ${}^{\perp_{0,1}}\Pcal$ for a collection $\Pcal$ of indecomposable pure-injective $\Lambda$-modules?}
\end{que}
Indeed, since any complex $X\in\Dcal(\lMod\Lambda)$ can be written as $X=\coprod_{n\in\mathbb{Z}} H^n(X)[-n]$,  we have that $X$ is in $^{\perp_{\mathbb{Z}}}\Pcal$ if and only if so are all its cohomologies, or equivalently, 
all its cohomologies belong to ${}^{\perp_{0,1}}\Pcal$. 
Hence $^{\perp_{\mathbb{Z}}}\Pcal$ and  ${}^{\perp_{0,1}}\Pcal$ correspond to each other under the bijection in Proposition~\ref{widecoref}.

\medskip

From now on $\Lambda$ denotes the Kronecker algebra, and we use the notation from Example~\ref{betalpha} and Section~\ref{tame}. 
In the Kronecker case, the proof of Theorem~\ref{tamethm} is much easier, as %Let us investigate the left perpendicular categories of the form ${}^{\perp_{0,1}}\Pcal$. 
%\begin{thm} \label{perppi}Let $(\Xcal, \Ycal)$ be a complete Ext-orthogonal pair over the Kronecker algebra $\Lambda$. The following statements are equivalent.\begin{enumerate}\item There is a set of $\Pcal$ of indecomposable pure-injective $\Lambda$-modules such that $\Xcal={}^{\perp_{0,1}}\Pcal$.\item  $\Xcal$ or $\Ycal$ is bireflective.\item There is $\Wcal\in\mathbf{wide}(\Lambda)$ such that $\Xcal=\Wcal^{\perp_{0,1}}$ or $\Xcal=\varinjlim \Wcal$.\end{enumerate}\end{thm}
the relevant classes can be computed directly. We present this alternative proof for the reader's convenience.
As usual, $\tau$  denotes the Auslander-Reiten translation, and  $\Gamma_\Lambda$ is the Auslander-Reiten quiver of $\Lambda$.

\begin{lem}
\label{prinj}
Let $M$ be a module in $\mathbf{p}\cup\mathbf{q}$. Then ${}^{\perp_{0,1}}M=\Add N$ where $N$ is the successor of $M$ in $\Gamma_\Lambda$ or $N$ is simple projective. Moreover, $\Wcal=\add M$ is a wide subcategory of $\lmod\Lambda$ with
\begin{enumerate}
\item $\Wcal^{\perp_{0,1}}=\Add L={}^{\perp_{0,1}}K$ where $L$ is the predeccessor of $M$ in $\Gamma_\Lambda$ or $L$ is simple injective, and $K=\tau M$ or $K$ is indecomposable injective.

%{}^{\perp_{0,1}}N$ for $N=\tau M$ or $N$ indecomposable injective.
\item $\varinjlim\Wcal=\Add M={}^{\perp_{0,1}}L$ where $L$ is the predeccessor of $M$ in $\Gamma_\Lambda$ or $L$ is simple injective.
%\item  ${}^{\perp_{0,1}}M=N^{\perp_{0,1}}$ for $N=\tau^- M$ or $N$ indecomposable projective.
\end{enumerate} 
\end{lem}
\begin{proof} $\Wcal=\add M$ is a wide subcategory because $M$ is a stone, i.e.~a brick without self-extensions. The remaining statements are easy observations obtained from the shape of  $\Gamma_\Lambda$.
\end{proof}

\begin{lem}
\label{reg}
Let $\emptyset\neq P\subset \mathbb X$ and $Q=\mathbb X\setminus P$.    Moreover, let $\Pcal$ be the set of adic modules corresponding to the simple regulars in $\mathbf{t}_P$, and let $\Qcal$ be the set consisting of the generic module $G$ and the Pr\"ufer modules corresponding to the simple regulars in $\mathbf{t}_Q$.  Then  $\Wcal=\add\mathbf{t}_P$ is a wide subcategory of $\lmod\Lambda$ with
 $\Wcal^{\perp_{0,1}}={}^{\perp_{0,1}}\Pcal$ and $\varinjlim\Wcal={}^{\perp_{0,1}}\Qcal$.
\end{lem}
\begin{proof}
Consider the cosilting torsion pair $(\Tcal,\Fcal)=(\Gen\mathbf{t}_P,\Fcal_P)=(T(\Wcal),\Wcal^{\perp_0})$. It is in fact a cotilting torsion pair  with minimal approximation sequence $0\to C_1\to C_0\to E(\Lambda)\to 0$ where $\Prod C_1=\Prod\Pcal$, and $\Prod C_0=\Prod\Qcal$.
% For details we refer to \cite{BK}.
Combining Remark~\ref{perpendicular} with Proposition~\ref{prop:alphaIsLimClosure} and Lemma~\ref{lem:betaCompHereditary} we obtain  $\varinjlim\Wcal=\alpha(\Tcal)={}^{\perp_{0,1}}C_0={}^{\perp_{0,1}}\Qcal$ and $\Wcal^{\perp_{0,1}}=\beta(\Fcal)={}^{\perp_{0,1}}C_1={}^{\perp_{0,1}}\Pcal$. 
\end{proof}

\noindent\emph{\bf Proof of Theorem~\ref{tamethm} in the Kronecker case.}
%The equivalence of (2) and (3) follows immediately from Theorem~\ref{wide and extorth}. The  equivalence of (1) and (3) will be proved by classification.
In order to show (1)$\Rightarrow$(2), we use the following table which summarizes Lemma~\ref{prinj} and~\ref{reg}. 

%%%
$$\begin{array}{cccccccccl}
\Wcal &&&& \Wcal^{\perp_{0,1}}&&&&\varinjlim\Wcal\\
\hline
0&&&& \lMod\Lambda={}^{\perp_{0,1}}\emptyset &&&&0={}^{\perp_{0,1}}\{\text{all indec.~pure-inj.}\}\\
\lmod\Lambda &&&& 0={}^{\perp_{0,1}}\{\text{all indec.~pure-inj.}\}&&&&\lMod\Lambda={}^{\perp_{0,1}}\emptyset \\
\add M,\, M\in\mathbf{p}\cup\mathbf{q}&&&& ^{\perp_{0,1}}K \;\text{for suitable}\; K &&&&^{\perp_{0,1}}L \;\text{for suitable}\; L\\
\add\mathbf{t}_P,\,\emptyset\neq P\subset\mathbb X &&&& ^{\perp_{0,1}}\{\text{adics  from}\; P\} &&&&^{\perp_{0,1}}\{G,\;\text{Pr\"ufer  from}\; Q=\mathbb X\setminus P\}
\end{array}$$
\\

\noindent
For the implication (2)$\Rightarrow$(1), we start by collecting the basic situations in the following table.
%%%

$$\begin{array}{cccccccccl}
\Pcal &&&& \Xcal={}^{\perp_{0,1}}\Pcal&&&&\Wcal\;\text{with}\;\Xcal=\Wcal^{\perp_{0,1}}\,\text{or}\,\Xcal=\varinjlim\Wcal\\
\hline
\emptyset &&&& \lMod\Lambda&&&&\lmod\Lambda\\
\text{all indec.~pure-inj.} &&&& 0&&&&0\\
M\in\mathbf{p}\cup\mathbf{q}&&&& \Add N \;\text{for suitable}\; N\in\mathbf{p}\cup\mathbf{q}&&&&\add N\\
M\in \mathbf{t}_x &&&& \mathbf{t}_x\,^{\perp_{0,1}} &&&&\add\mathbf{t}_x\\
\text{adics from}\; P &&&& \mathbf{t}_P\,^{\perp_{0,1}}&&&&\add\mathbf{t}_P \\\
 
\text{Pr\"ufer from}\; Q &&&& \varinjlim\add\mathbf{t}_P \;\text{with}\; P=\mathbb X\setminus Q&&&&\add\mathbf{t}_P\\
G &&&& \varinjlim\add\mathbf{t} &&&&\add\mathbf{t}
\end{array}$$
\\

%%%
We only need to  explain line 4, since the other cases follow immediately from Lemma~\ref{G},\ref{prinj} and~\ref{reg}. For line 4, we note that if $\Pcal=\{M\}$ with $M\in\mathbf{t}_x$, then $\Ycal$ contains the wide closure of $M$, that is $\add\mathbf{t}_x$, and thus $\Xcal={}^{\perp_{0,1}}\mathbf{t}_x$, which coincides with $\mathbf{t}_x\,^{\perp_{0,1}}$ by the Auslander-Reiten formula. 
%For line 7, we refer to \cite[Section 3.4]{RR} stating that ${}^{\perp_{0,1}}G={}^{\perp_{0}}G\cap{}^{\perp_{1}}G=\Gen\mathbf{t}\cap\,\Ccal$  coincides with $\varinjlim\add\mathbf{t}$.

Now, let us consider an arbitrary set of indecomposable pure-injectives $\Pcal$. First, we see that $\Xcal=0$ provided that $\Pcal$ contains more than one finite dimensional module. Similarly, $\Xcal=0$ whenever $\Pcal$ contains a module $M\in\mathbf{p}\cup\mathbf{q}$ together with an infinite dimensional  indecomposable pure-injective, because ${}^{\perp_{0,1}}M=\Add N$ for a suitable $N\in\mathbf{p}\cup\mathbf{q}$, and $N$ does neither belong to $\varinjlim \add\mathbf{t}_P$ nor to $ \mathbf{t}_P\,^{\perp_{0,1}}$ for any $P\subset\mathbb X$. 
For the remaining cases, it is enough to observe that $ \mathbf{t}_P\,^{\perp_{0,1}}\cap\,\varinjlim\add\mathbf{t}_Q=\varinjlim\add\mathbf{t}_{Q\setminus P}$. 
Altogether, we can conclude that in all cases $\Xcal$ arises as $\Xcal=\Wcal^{\perp_{0,1}}$ or $\Xcal=\varinjlim \Wcal$ for some $\Wcal\in\mathbf{wide}(\Lambda)$.
\hfill$\Box$

\bigskip

%%%%%%%%%%%%%%%%%%%%%%%%

In \cite{tameWild}, Ringel constructs a family of bricks $P(I)$ over $\Lambda$ indexed by the subsets of the ground field $k$. Recall that the generic module $G$ corresponds to the representation 
$$\begin{tikzcd}
k(T) \ar[r, shift left=1] \ar[r, shift right=1] & k(T)
\end{tikzcd}$$ 
given by the field of fractions $k(T)$ together with the identity map and the multiplication $T\cdot$ by the element $T$. The module $P(I)$ is constructed as  the subrepresentation 
$$\begin{tikzcd}
V(I) \ar[r, shift left=1] \ar[r, shift right=1] & V(I)+ k\cdot 1\end{tikzcd} $$
where $V(I)$ is the vectorspace with basis $\{\frac{1}{T-\lambda}\mid\,\lambda \in I\}$, and $k\cdot 1$ is the one-dimensional vectorspace generated by the element $1\in k(T)$. 
We collect some properties of these modules. 

\begin{lem}\cite{tameWild}\label{largebricks}
Let  $I$ be a subset of $k$ and  $S(I)=\bigoplus\limits_{\lambda\in I} S_\lambda$ where $S_\lambda$ is the simple regular in  ${\mathbf t}_\lambda$.
\begin{enumerate}
\item When $I$ is a set of cardinality $n$, then $P(I)$ is  indecomposable preprojective of dimension vector $(n,n+1)$. In particular, $P:=P(\emptyset)$ is the simple projective $\Lambda$-module.
\item When $I$ is an infinite set, $P(I)$ is an infinite dimensional brick.
\item $\Hom_\Lambda(P(I),P(J))=0$ whenever $I,J$ are two infinite disjoint sets.
\item For  any subset $J\subset I$ there is a short exact sequence $0\to P(J)\to P(I)\to S(I\setminus J)\to 0$.
% and  $0\to B(J)\to B(I)\to \bigoplus\limits_{\lambda\in I\setminus J} S_\lambda$.
\end{enumerate}
\end{lem}
 
We now use these large bricks to construct a wide subcategory of $\lMod\Lambda$ which might not fit in the classification from Theorem~\ref{tamethm}.

\begin{prop}\label{perpB}
Let  $B=P(I)$ be constructed as above from an infinite subset $I\subset k$, and  let $(\Tcal,\Fcal)=({}^{\perp_0}B,F(B))$ be the torsion pair cogenerated by $B$.
Then
\begin{enumerate}
\item $\Hom_\Lambda(B,S_\lambda)\ne 0$ for all $\lambda\in k\cup\{\infty\}$.
\item $\alpha(\Tcal)={}^{\perp_{0,1}}B$ has no nonzero finite dimensional modules.
\end{enumerate}
\end{prop}
\begin{proof}
(1)  The sequence $0\to P\to B\to S(I)\to 0$ from Lemma~\ref{largebricks}(4)  shows that  statement for $\lambda\in I$. For $\lambda\in k\setminus I$ we use the non-split exact sequence $0\to B\to P(I\cup\{\lambda\})\to S_\lambda\to 0$  to see that $\Hom_\Lambda(B,S_\lambda)\cong\Ext^1_\Lambda(S_\lambda,B)\not=0$. 
It remains to show $\Hom_\Lambda(B,S_\infty)\ne 0$. To this end, we regard $S_\infty$ as representation 
$\begin{tikzcd}
k\cdot 1 \ar[r, shift left=2] \ar[r] & k\cdot 1\end{tikzcd}$
given by the linear map $0$ and the identity map $\textrm{id}_{k\cdot 1}$, 
and we define a linear map $f:V(I)\longrightarrow k\cdot 1$ on the basis $\{\frac{1}{T-\lambda}\mid\,\lambda \in I\}$ of $V(I)$ by setting $f(\frac{1}{T-\lambda})=1$ for all $\lambda\in I$. Since the elements $\frac{1}{T-\lambda},\,\lambda \in I,$ and $1$ are linearly independent in $k(T)$,  we can further define a linear map $g: V(I)+ k\cdot 1\longrightarrow k\cdot 1$ by setting $g\mid_{V(I)}=0$ and $g(1)=1$. Now we have  $g\circ\textrm{id}\mid_{V(I)}=0 \circ f$,  and $g(T\cdot \frac{1}{T-\lambda})= g(\lambda \cdot \frac{1}{T-\lambda} + 1)= g(1)=f(\frac{1}{T-\lambda})$ for all $\lambda\in I$, that is, $g\circ{T\cdot\mid_{V(I)}}= \textrm{id}_{k\cdot 1}\circ f$. Thus $f$ and $g$ define a non-zero morphism of representations  $B\to S_\infty$.

(2) We show that ${}^{\perp_{0,1}}B\subseteq\alpha(\Tcal)$. If $g:T\to X$ is a morphism with $T\in\Tcal$ and $X\in{}^{\perp_{0,1}}B$, then its image obviously lies in $\Tcal$, and even in ${}^{\perp_{0,1}}B$, as ${}^{\perp_1}B$ is closed under submodules. Thus we can assume without loss of generality that $g$ is surjective. Now applying $\Hom_\Lambda(-,B)$ on the exact sequence $0\to K\to T\stackrel{g}{\rightarrow} X\to 0$ and using that  $\Hom_\Lambda(T,B)=\Ext^1_\Lambda(X,B)=0$ we conclude that $K\in\Tcal$.

For the reverse inclusion we have to show that every $X\in\alpha(\Tcal)$ satisfies $\Ext^1_\Lambda(X,B)=0$. Consider a short exact sequence $0\to B\stackrel{f}{\rightarrow}  E\stackrel{g}{\rightarrow} X\to 0$. The middle term $E$ cannot belong to $\Tcal$, otherwise $B\in\Tcal\cap\Fcal=0$. So its torsion-free part $\overline{E}=E/t(E)$ is non-zero and thus admits a non-zero map $h:\overline{E}\to B$. Then the composition $h\nu$ of $h$ with the canonical epimorphism $\nu:E\to \overline{E}$ is also non-zero. Now suppose that $h\nu f=0$. Then $h\nu$ factors through $g$, that is, $h\nu=\overline{h} g$ for some non-zero map $\overline{h}:X\to B$, contradicting the hypothesis $X\in \Tcal$. We infer that $h\nu f:B\to B$ is non-zero and therefore an isomorphism. This shows that our exact sequence splits, as desired.

%$\Xcal$ is clearly closed under extensions. Moreover, since ${}^{\perp_0}B$ is closed under quotients and ${}^{\perp_1}B$ under submodules, it is clear that $\Xcal$ is closed under images. Then it is enough to check that $\Xcal$ is closed under kernels of epimorphisms and cokernels of monomorphisms, and this follows easily by applying $\Hom_\Lambda(-,B)$ and keeping in mind that $\Ext^2_\Lambda(-, B)=0$.

Now recall that $B$ is an infinite dimensional brick, and in particular, $B$ has no direct summands in 
$\textbf{p}\cup\textbf{q}$. Together with (1), this  shows that no preprojective module belongs to ${}^{\perp_0}B$ and that neither preinjective nor regular modules can belong to ${}^{\perp_1}B$. Hence $\alpha(\Tcal)\cap\lmod \Lambda=0$. 
\end{proof}

%would imply that $B$ belongs to $(\Cogen G)^{\perp_1}$. Now recall that $\Fcal=\Cogen G$ is a cotilting class, and the module $C$ given as the direct product of $G$ and all adic modules is a cotilting module cogenerating $\Fcal$.  So $B$ would belong to $\Fcal\cap\Fcal^{\perp_1}=\Prod C$ and would thus have to be   either isomorphic to $G$ or to an adic module. But the first case is not possible because $B$ is properly contained in $G$, and in the second case we would have $\Hom_\Lambda(B,S_\lambda)=0$ for all but one $\lambda\in k\cup\{\infty\}$, which would also yield a contradiction.

\begin{prop}\label{sets} 
%Assume either $V=L$, or Vopenka's principle.
Let  $B=P(I)$ be constructed as above from an infinite subset $I\subset k$ which is not cofinite, and  let $(\Tcal,\Fcal)=({}^{\perp_0}B,F(B))$ be the torsion pair cogenerated by $B$. Assume that  $\alpha(\Tcal)\not=0$. Then there is   a wide and coreflective subcategory $\Xcal$ of  $\lMod \Lambda$ which is contained in $\alpha(\Tcal)$ and is not of the form $\Xcal=\Wcal^{\perp_{0,1}}$ nor $\Xcal=\varinjlim \Wcal$ for some wide subcategory $\Wcal\in\mathbf{wide}(\Lambda)$.
\end{prop}
\begin{proof}
%By \cite[Theorem 10.1.4]{approximationsEndo}, the set-theoretic axiom $V=L$ ensures that every module $M\in\lMod\Lambda$  has a right ${}^{\perp_1}B$-approximation $X\to M$. Taking the torsion part of $X$ with respect to the torsion pair  $(\Tcal,\Fcal)=({}^{\perp_0}B,F(B))$, we obtain that the composition $t(X)\hookrightarrow X\to M$ is a right $\Xcal$-approximation of $M$. This show that $\Xcal$ is precovering and closed under cokernels, hence coreflective by Lemma~\ref{coref}.
%Also  Vopenka's principle  implies that $\Xcal$ is coreflective, albeit for a different reason.  Indeed, this   set-theoretic axiom ensures that every subcategory closed under coproducts and directed colimits is covering, see \cite[Theorem 5.33]{approximationsEndo}. So every wide subcategory closed under coproducts is coreflective.
%Now we know from Proposition~\ref{widecoref} that the category $\Xcal$ occurs in a complete Ext-orthogonal pair $(\Xcal,\Ycal)$, and according to Propositions~\ref{prop:ExtOrthEmbedding} and~\ref{perpB}, \reml{??}\chl{the associated torsion pair is $(\Tcal,\Fcal)$}. This torsion pair is non-trivial because $0\ne\Fcal=F(B)\subset\Cogen G$. Therefore $0\ne\Xcal\ne\lMod\Lambda$. 
We  claim that $\alpha(\Tcal)$ does not contain any indecomposable pure-injective module. We already know from Proposition~\ref{perpB} that $\alpha(\Tcal)$ does not contain indecomposable finite dimensional modules. Since every simple regular module occurs as kernel of an endomorphism  of the associated Pr\"ufer module, $\alpha(\Tcal)$ cannot contain any Pr\"ufer module, and dually, it cannot contain any adic module.
We now show that it cannot contain the generic module $G$. 

To see this, we first observe that
% ${}^{\perp_{0,1}}B\subset\Cogen G$. In fact, 
every  $X\in{}^{\perp_{1}}B$ lies in $\Cogen G$. Indeed, $X$ can be  written as a direct limit of its finitely generated submodules, which lie again in ${}^{\perp_{1}}B$ and are therefore  preprojective. Thus $X\in\varinjlim\add\mathbf{p}=\Cogen G$. 
In particular, $\alpha(\Tcal)={}^{\perp_{0,1}}B$ is contained in $\Cogen G$.

Next, we assume that $G$ belongs to $\alpha(\Tcal)$ and take a nonzero subobject $X$ in $\alpha(\Tcal)$. Then $G/X$ is in $\alpha(\Tcal)$ and therefore admits an embedding in a product of copies of $G$. Since $G$ is a brick and $G\to G/X$ is a proper epimorphism, we conclude that $X=G$. This shows that $G$ must be  a simple object in $\alpha(\Tcal)$. It follows from Proposition~\ref{prop:simpInalphabeta} that $G$ is  torsion, almost torsion-free with respect to $(\Tcal,\Fcal)$. But this is not possible, because any module of the form $P(J)$ with $J$ an infinite subset of $k$ disjoint from $I$ is a proper submodule of $G$ which lies in $\Tcal$ by Lemma~\ref{largebricks}(3). This concludes the proof of our claim.

Now we assume there is an object $0\not=X\in\alpha(\Tcal)$. We set $\Xcal$ to be the smallest wide subcategory  of  $\lMod \Lambda$ which is closed under coproducts and contains $X$. Then $\Xcal$ is coreflective by \cite[Theorem 2.2]{extPairs}, and it is a subcategory of $\alpha(\Tcal)$ by construction. 

Of course, $\Xcal$ does not contain any indecomposable pure-injective module. 
In particular, $\Xcal\cap\lmod \Lambda=0$, and we immediately see that  $\Xcal$ is  not of the form $\Wcal^{\perp_{0,1}}$ for some $\Wcal=\add M$ with $M\in\mathbf{p}\cup\mathbf{q}$, nor for some $\Wcal=\add\mathbf{t}_P$ arising from a proper subset $P$ of $\mathbb X$. 
Moreover, we can also exclude that $\Xcal=(\add\mathbf{t})^{\perp_{0,1}}=\Add G$. 
Finally, $\Xcal$ can't be of the form $\varinjlim\Wcal$ for some $\Wcal\in\mathbf{wide}(\Lambda)$, because 
this would imply $\Wcal=0$ by Theorem~\ref{themapalpha}, contradicting the hypothesis $\Xcal\not=0$.
\end{proof}

\end{document}